\pdfoutput=1
\RequirePackage{ifpdf}
\ifpdf 
\documentclass[pdftex]{sigma}
\else
\documentclass{sigma}
\fi

\usepackage{ytableau,tikz}

\numberwithin{equation}{section}

\newtheorem{Theorem}{Theorem}[section]
\newtheorem{Corollary}[Theorem]{Corollary}
\newtheorem{Lemma}[Theorem]{Lemma}
\newtheorem{Proposition}[Theorem]{Proposition}

{ \theoremstyle{definition}
\newtheorem{Definition}[Theorem]{Definition}

 }

\newcommand{\tab}{\textsf{Tab}}

\renewcommand{\S}{\mathbb{S}}
\newcommand{\ct}[2]{\textsf{CT}_{\mathbb{#1}}[#2]}
\renewcommand{\c}[1]{\textsf{CT}_{\mathbb{#1}}}

\begin{document}
\allowdisplaybreaks

\newcommand{\arXivNumber}{1909.00071}

\renewcommand{\thefootnote}{}

\renewcommand{\PaperNumber}{010}

\FirstPageHeading

\ShortArticleName{Singular Nonsymmetric Macdonald Polynomials and Quasistaircases}

\ArticleName{Singular Nonsymmetric Macdonald Polynomials\\ and Quasistaircases}

\Author{Laura COLMENAREJO~$^\dag$ and Charles F.~DUNKL~$^\ddag$}

\AuthorNameForHeading{L.~Colmenarejo and C.F.~Dunkl}

\Address{$^\dag$~Department of Mathematics and Statistics, University of Massachusetts at Amherst,\\
\hphantom{$^\dag$}~Amherst, USA}
\EmailD{\href{mailto:laura.colmenarejo.hernando@gmail.com}{laura.colmenarejo.hernando@gmail.com}}
\URLaddressD{\url{https://sites.google.com/view/l-colmenarejo/home}}

\Address{$^\ddag$~Department of Mathematics, University of Virginia, Charlottesville VA 22904-4137, USA}
\EmailD{\href{mailto:cfd5z@virginia.edu}{cfd5z@virginia.edu}}
\URLaddressD{\url{http://people.virginia.edu/~cfd5z/}}

\ArticleDates{Received September 06, 2019, in final form February 19, 2020; Published online February 27, 2020}

\Abstract{Singular nonsymmetric Macdonald polynomials are constructed by use of the representation theory of the Hecke algebras of the symmetric groups. These polynomials are labeled by quasistaircase partitions and are associated to special parameter values $(q,t)$. For $N$ variables, there are singular polynomials for any pair of positive integers $m$ and $n$, with $2\leq n\leq N$, and parameters values $(q,t)$ satisfying $q^{a}t^{b}=1$ exactly when $a=rm$ and $b=rn$, for some integer $r$. The coefficients of nonsymmetric Macdonald polynomials with respect to the basis of monomials $\big\{ x^{\alpha}\big\}$ are rational functions of $q$ and $t$. In this paper, we present the construction of subspaces of singular nonsymmetric Macdonald polynomials specialized to particular values of $(q,t)$. The key part of this construction is to show the coefficients have no poles at the special values of $(q,t)$. Moreover, this subspace of singular Macdonald polynomials for the special values of the parameters is an irreducible module for the Hecke algebra of type $A_{N-1}$.}


\Classification{33D52; 20C08; 33D80; 05E10}

\section{Introduction}

The Hecke algebra $\mathcal{H}_{N}(t)$ of the symmetric group $\mathcal{S}_{N}$ acting on $\{1,2,\ldots,N\}$ has representations on polynomials in $N$ variables as well as on finite-dimensional spaces spanned by reverse standard Young tableaux (RSYT) of shape~$\tau$, for each partition~$\tau$ of~$N$. Among the different polynomials related to the Hecke algebra, the nonsymmetric Macdonald polynomials are defined as homogeneous eigenvectors of the Cherednik operators.

In any structure of algebra and analysis that involves parameters, it is always crucial to know the effect of different parameter values, for instance, when shifted nonsymmetric Macdonald polynomials become homogeneous (see \cite[Proposition~2, p.~9]{DL2015}). Here we are concerned with parameters giving rise to singular nonsymmetric Macdonald polynomials. We analyze the situations where the Cherednik operators coincide with Jucys--Murphy elements of the Hecke algebra. It is remarkable that this leads directly to singular polynomials, which are defined to be in the joint kernels of Dunkl operators. We already looked at singular Macdonald polynomials in our work with Jean-Gabriel Luque in~\cite{CDL2019}, where the singular polynomials form the basic ingredient of the projection map described there.

In this paper we construct spaces of nonsymmetric Macdonald polynomials which admit a~representation isomorphic to the representation on finite-dimensional spaces spanned by RSYT for certain shape $\tau$ and parameter values $(q,t)$. Furthermore, the partitions that arise are related to quasistaircases. As a very initial example, let $N=10$ and consider the quasistaircase partition $\lambda= (4,3,3,2,2,0,0,0,0,0)$. In this case, we will look at RSYT of shape $\tau=(5,2,2,1)$ and parameter values satisfying $qt^{3}=1$.

It is known that the quasistaircases exhaust all singular nonsymmetric Jack polynomials and we suspect that this also holds for singular nonsymmetric Macdonald polynomials~\cite{D05}. By use of quasistaircases we will construct these subspaces. The idea is that, once we fix certain partition~$\tau$ and parameter values $(q,t)$, for each RSYT of shape $\tau$, denoted by~$\S$, there is a~label~$\alpha(\S)$ such that the associated set of nonsymmetric Macdonald polynomials $\{M_{\alpha(\S)}\}$ is a basis of isotype~$\tau$ and the spectral vectors satisfy that $\zeta_{\alpha(\S)}=\big[t^{\ct{S}{i}}\big]_{i=1}^{N}$.
The partition $\tau$ will be of the form $\tau=\big( dn-1,(n-1)^{K-1},\nu_{K}\big)$, for some specific parameters $d$, $n$, $K$, and $\nu_K$, and the specialization of the parameter will be of the form $\varpi=(q,t)=\big(\omega u^{-n/g},u^{m/g}\big)$, where $m$ and $n$ are integers such that $g=\gcd(m,n)$ and $u$ is not a root of unity and $u\neq0$. With this idea in mind, we state our main theorem now.

\begin{Theorem}\label{MainThm}The polynomials $\{ M_{\alpha(\S )}\colon \S \in \mathsf{Tab}_{\tau}\} $ specialized to $(q,t)=\varpi$ are a~basis of isotype $\tau$ and are singular.
\end{Theorem}

This paper covers an explanation of all the concepts involved in Theorem~\ref{MainThm}, as well as the presentation of its proof. The presentation begins in Section~\ref{Sec2} with a concise overview of the background needed in this paper. This section includes some combinatorial definitions, together with an exposition of the representation theory of the Hecke algebra and of nonsymmetric Macdonald polynomials and singular polynomials. In Section~\ref{QUStair}, we introduce the quasistaircase partitions and the specialization that we will be considering through this paper. Section~\ref{Sec4} is dedicated to introduce the concept of the \emph{equipolar property} since it will simplify notably our study. We warn the reader that the sketch of the proof of Theorem~\ref{MainThm} is included in Section~\ref{SecProofMainThm}. The rest of the paper is dedicated to prove some technical results. In Section~\ref{Sec5}, we use the \emph{critical pair method} and we present the minimal set of configurations that need to be checked.
This is done in Section~\ref{Sec6}, where we finish our study by carefully analyzing the critical pairs for the quasistaircase partitions. Finally, we wrap up the paper with some concluding remarks and an illustrative example in Section~\ref{Sec7}.

\section{Background}\label{Sec2}

This paper relates concepts and uses notation from different areas. In this section we set up the foundations and the notation by reviewing the basic definitions and results that are involved in our study. The section is split according to the different areas.

\subsection{Combinatorics}

Let us start with the combinatorial objects. For more details, see~\cite{Macd1995, S1999}.

A \emph{partition} $\tau=(\tau_{1},\dots,\tau_{N})$ is a nonincreasing sequence such that $\tau_{i}\geq0$, for all $i$. The length of a partition $\tau$ is the number of nonzero parts of $\tau$, $\ell(\tau)=\max \{i\colon \tau_{i}>0\}$. Moreover, we say that $\tau$ is a partition of $n$, or that the size of $\tau$ is $n$, if $\sum_{i} \tau_{i}=n$. We denote by $\tau\vdash n$ or $|\tau|=n$ if $\tau$ is a partition of $n$ and by $\mathsf{Par}(n)$ the set of partitions of $n$. We consider the following \emph{partial order} on partitions. For $\tau,\gamma\in \mathsf{Par}(n)$, we say that $\tau$ dominates $\gamma$, and we write $\tau\succ\gamma$, if $\tau\neq\gamma$ and $\sum\limits_{i=1}^{j}\tau_{i} \geq\sum\limits_{i=1}^{j} \gamma_{i}$, for all $1\leq j\leq n$.

A \emph{composition} $\alpha=(\alpha_{1},\dots,\alpha_{N})$ is any permutation of a partition. We denote by $\alpha^{+}$ the unique nonincreasing rearrangement of $\alpha$ such that $\alpha^{+}$ is a partition. We say that~$\alpha$ is a~composition of $n$, or that has size $n$, if $|\alpha|=|\alpha^{+}|=n$.

The definition of the partial order on partitions applies also for compositions since it does not use that the sequences are weakly decreasing. We continue using the symbol $\succ$ for this order for compositions.
Moreover, it can be used to define another order. For $\alpha$ and $\beta$ compositions, we write $\alpha\triangleright\beta$ if $|\alpha|=|\beta|$, $\alpha\neq\beta$, and either $\alpha^{+}\succ\beta^{+}$, or $\alpha^{+}=\beta^{+}$ and $\alpha\succ\beta$.

\begin{remark*}Notice that, by definition, the partitions and compositions appearing in this paper are allowed to have zeros and are standardized to have $N$ entries in total (including the zeros). However, we omit the zero entries in those partitions for which they are not relevant. We mostly work with $\mathsf{Par} (N)$, the set of partitions $\tau=(\tau_{1},\dots,\tau_{N})$ with $\sum_{i} \tau_{i} = N$.
\end{remark*}

Given a composition $\alpha$, we associate to it a \emph{rank function} $r_{\alpha}= (r_\alpha(1),r_\alpha(2),\dots, r_\alpha(N) )$ by setting
\begin{gather}\label{Def:rankfunct}
r_{\alpha}(i)=\# \{ k \colon 1\leq k\leq N,\,\alpha_{k}>\alpha_{i} \} +\# \{ k \colon 1\leq k\leq i,\,\alpha_{k}=\alpha_{i}\},
\end{gather}
for $1\leq i\leq N$, where we use the notation $\# A$ to denote the size of the set $A$. It is important to point out that $r_{\alpha}$ is a permutation of $\{1,2,\dots,N\}$. Moreover, $r_{\alpha}=(1,2,\dots,N)$ if and only if~$\alpha$ is a partition. Therefore, $\alpha^{+}$ satisfies that $\alpha
_{r_{\alpha}(i)}^{+}=\alpha_{i}$, for $1\leq i \leq N$.

A \emph{Ferrers diagram of shape} $\tau\in\mathsf{Par}(n)$ is obtained by drawing $\tau_i$ boxes from bottom to top, all shifted to the left (corresponding to French notation). That is, we draw boxes at points $(i,j)$, for $1\leq i \leq\ell(\tau)$ and $1\leq j\leq\tau_{i}$, in the $xy$-plane. We define two fillings of a Ferrers diagram of shape $\tau\in\mathsf{Par}(n)$. A \emph{reverse standard Young tableau} (RSYT) is a filling such that the entries are exactly $\{1,2,\dots,n\}$ and are decreasing in rows and columns
when reading from left to right and from bottom to top.
A \emph{reverse row-ordered standard Young tableau} is a filling such that the entries are exactly $\{1,2,\dots,n\}$ and are decreasing in rows, with no condition on the columns. Our main objects are the RSYT, and therefore we denote by $\mathsf{Tab}_{\tau}$ the set of RSYT of shape $\tau$ and by
$V_{\tau}$ the space with orthogonal basis given by $\mathsf{Tab}_{\tau}$, i.e., $V_{\tau}=\mathrm{span}_{\mathbb{R}(t)} \{ \mathbb{S}\colon \mathbb{S\in}
\mathrm{Tab}_{\tau} \}$. We also denote by $\mathsf{RSTab}_{\tau}$ the set of reverse row-ordered standard Young tableaux of shape $\tau$. Note that $\mathsf{Tab}_{\tau}\subset
\mathsf{RSTab}_{\tau}$.

We finish this subsection introducing useful notation for the tableaux in $\mathsf{Tab}_{\tau}$. Let $\S \in\mathsf{Tab}_{\tau}$, for some partition $\tau \vdash N$. The entry $i$ of $\S$ is at \emph{coordinates} $(\mathsf{row}_{\mathbb{S}}[i],\mathsf{col}_{\mathbb{S}}[i])$, where $\mathsf{row}_\S[i]$ denotes the row in which $i$ appears (counting from bottom to top) and $\mathsf{col}_\S[i]$ denotes the column in which $i$ appears (counting from left to right). Moreover, the \emph{content} of the entry is $\mathsf{CT}_{\mathbb{S}}[i]=
\mathsf{col}_{\mathbb{S}}[i]-\mathsf{row}_{\mathbb{S}}[i]$. Then, each $\S \in\mathsf{Tab}_{\tau}$ is uniquely determined by its \emph{content vector} $\c{S} = [\mathsf{CT}_{\mathbb{S}}[i]]_{i=1}^{N}$.
For instance, $\S =\scalebox{0.7}{
\begin{tikzpicture}
\draw (0,0) rectangle (0.5,0.5);
\draw (0,0.5) rectangle (0.5,1);
\draw (0.5,0) rectangle (1,0.5);
\draw (0.5,0.5) rectangle (1,1);
\draw (1,0) rectangle (1.5,0.5);
\draw (1,0.5) rectangle (1.5,1);
\draw (1.5,0) rectangle (2,0.5);
\node at (0.25,0.25) {7};
\node at (0.75,0.25) {6};
\node at (1.25,0.25) {5};
\node at (1.75,0.25) {2};
\node at (0.25,0.75) {4};
\node at (0.75,0.75) {3};
\node at (1.25,0.75) {1};
\end{tikzpicture}}$\ has shape $\tau=(4,3)$ and content vector $\c{S}=[1,3,0,-1,2,1,0]$.

Given $\S \in\mathsf{Tab}_{\tau}$, we define $\S^{(i)}$ to be the RSYT obtained by exchanging $i$ and $i+1$ in the case that $\mathsf{row}_{\mathbb{S}}[i]<\mathsf{row}_{\mathbb{S}}[i+1]$ and $\mathsf{col}_{\mathbb{S}}[i]>\mathsf{col}_{\mathbb{S}}[i+1]$. We refer this map $\S \xrightarrow{s_i} \S^{(i)}$ as a
\emph{step}. We reserve the notation $\S s_{i}$, which again exchange $i$ and $i+1$, for the case in which $\S s_{i}$ is not a RSYT. We also set up that $d[i]=\mathsf{CT}_{\mathbb{S}}[i]-\mathsf{CT}_{\mathbb{S}}[i+1]$, since it will appear several times.

There is a partial order on $\mathsf{Tab}_{\tau}$ related to the inversion number:
\begin{gather*}
\text{inv}(\S) = \# \{ (i,j)\colon 1\leq i<j \leq N,\, \mathsf{row}_{\S}[i]<\mathsf{row}_{\S}[j] \}.
\end{gather*}
We denote by $\S_{0}$ the \emph{inv-maximal element} of $\mathsf{Tab}_{\tau}$, which has the numbers $N,N-1,\dots,1$ entered column-by-column, and by $\S_{1}$ the \emph{inv-minimal element} of $\mathsf{Tab}_{\tau}$, which has these numbers entered row-by-row.
Note that $\text{inv}(\S_{1})=0$ and that $\text{inv}
\big(\S^{(i)}\big)= \text{inv}(\S)-1$.
\begin{example*}
For the shape $(4,3)$, $\S_{0}=\scalebox{0.7}{
\begin{tikzpicture}
\draw (0,0) rectangle (0.5,0.5);
\draw (0,0.5) rectangle (0.5,1);
\draw (0.5,0) rectangle (1,0.5);
\draw (0.5,0.5) rectangle (1,1);
\draw (1,0) rectangle (1.5,0.5);
\draw (1,0.5) rectangle (1.5,1);
\draw (1.5,0) rectangle (2,0.5);
\node at (0.25,0.25) {7};
\node at (0.75,0.25) {5};
\node at (1.25,0.25) {3};
\node at (1.75,0.25) {1};
\node at (0.25,0.75) {6};
\node at (0.75,0.75) {4};
\node at (1.25,0.75) {2};
\end{tikzpicture}}$
 with $\text{inv}(\S_0)=6$, and $\S_{1}=\scalebox{0.7}{
\begin{tikzpicture}
\draw (0,0) rectangle (0.5,0.5);
\draw (0,0.5) rectangle (0.5,1);
\draw (0.5,0) rectangle (1,0.5);
\draw (0.5,0.5) rectangle (1,1);
\draw (1,0) rectangle (1.5,0.5);
\draw (1,0.5) rectangle (1.5,1);
\draw (1.5,0) rectangle (2,0.5);
\node at (0.25,0.25) {7};
\node at (0.75,0.25) {6};
\node at (1.25,0.25) {5};
\node at (1.75,0.25) {4};
\node at (0.25,0.75) {3};
\node at (0.75,0.75) {2};
\node at (1.25,0.75) {1};
\end{tikzpicture}}$ with $\text{inv}(\S_1)=0$.
\end{example*}

\subsection{The Hecke algebra and its representations}\label{SubSect:Heckealgebra}

Let $t$ be a formal parameter (or a complex number not a root of unity). The \emph{Hecke algebra}~$\mathcal{H}_{N}(t)$ is the associative algebra generated by $\{T_{1},T_{2},\ldots,T_{N-1}\}$ subject to the relations
\begin{alignat*}{3}
& (T_{i}+1)(T_{i}-t) = 0, \qquad&& \text{for} \quad 1\leq i \leq N-1,&\\
& T_{i}T_{i+1}T_{i} = T_{i+1}T_{i}T_{i+1}, \qquad && \text{for} \quad 1\leq i\leq N-2,&\\
& T_{i}T_{j} =T_{j}T_{i},\qquad && \text{for} \quad 1\leq i<j-1\leq N-2.&
\end{alignat*}

The irreducible modules of $\mathcal{H}_{N}(t)$ are indexed by partitions of $N$. In fact, there is a representation of $\mathcal{H}_{N}(t)$ on $V_{\tau}$, which we denote by $\tau$ (slight abuse of notation).

Following~\cite{DJ1986}, we describe the representation in terms of the actions of $T_{i}$ on the basis elements. For $\S \in\mathsf{Tab}_{\tau}$ and $i$, with $1\leq i <N$,
\begin{enumerate}\itemsep=0pt
\item[(I)] If $\mathsf{row}_{\mathbb{S}}[i]=\mathsf{row}_{\mathbb{S}}[i+1]$, then $\S \tau(T_{i})=t\S $.
\item[(II)] If $\mathsf{col}_{\mathbb{S}}[i]=\mathsf{col}_{\mathbb{S}}[i+1]$, then $\S \tau(T_{i})=-\S $.
\item[(III)] If $\mathsf{row}_{\mathbb{S}}[i]<\mathsf{row}_{\mathbb{S}}[i+1]$ and $\mathsf{col}_{\mathbb{S}}[i]>\mathsf{col}_{\mathbb{S}}[i+1]$, then ${\S \tau(T_{i}) = \S^{(i)}+\frac{t-1}{1-t^{-d[i]}}\S }$.
\item[(IV)] If $\mathsf{CT}_{\mathbb{S}}[i]-\mathsf{CT}_{\mathbb{S}}[i+1]\leq-2$, then $\S \tau(T_{i}) = \frac{t(t^{d[i]+1}-1)(t^{d[i]-1}-1)}{(t^{d[i]}-1)^{2}}\S^{(i)}+\frac{t^{d[i]}(t-1)}{t^{d[i]}-1}\S $.
\end{enumerate}

Observe that the last case can be obtained from Case (III) by
interchanging $\S $ and $\S^{(i)}$ and applying the relation $(\tau(T_{i})+I)(\tau(T_{i})-tI)=0$, where $I$ denotes the identity operator on $V_{\tau}$. We will refer to the formulas (I)--(IV) as the \emph{action formulas for $\tau(T_{i})$}.

Consider the following inner product on $V_{\tau}$. For $\S , \S^{\prime}\in\mathsf{Tab}_{\tau}$, $\langle\S ,\S^{\prime}\rangle_{t}=\delta_{\S ,\S^{\prime}}\cdot\gamma(\S ,t)$, with
\begin{gather*}
\gamma(\S ;t)= \prod_{\substack{i<j \\\mathsf{CT}_{\mathbb{S}}[j]-\mathsf{CT}_{\mathbb{S}}[i]\geq2}} \frac{\big(1- t^{\mathsf{CT}_{\mathbb{S}}[j]-\mathsf{CT}_{\mathbb{S}}[i]-1} \big)\big(1-t^{\mathsf{CT}_{\mathbb{S}}[j]-\mathsf{CT}_{\mathbb{S}}[i]+1}\big)} {\big(1-t^{\mathsf{CT}_{\mathbb{S}}[j]-\mathsf{CT}_{\mathbb{S}}[i]} \big)^{2}},
\end{gather*}
and extended by linearity. Note that this inner product satisfies that $\langle fT_{i},g\rangle = \langle f,gT_{i}\rangle $, for $f,g\in V_{\tau}$, and that it is invariant under the transformation ${t \longmapsto t^{-1}}$.

For $\mathcal{H}_{N}(t)$, a \emph{set of Jucys--Murphy elements} is defined by the following recursive formula:
\[
\begin{cases}
\phi_{N} = 1,\\
\phi_{i} = \dfrac{1}{t}T_{i}\phi_{i+1}T_{i},\hspace{0.5cm}\text{for }1\leq i<N.
\end{cases}
\]
In~\cite{DL2012}, there is described another set of Jucys--Murphy elements. The set described here is nicely linked to singularity and seems easier to manipulate in this setup. Next, we describe the action of this set of Jucys--Murphy elements on RSYT.

\begin{Proposition}\label{eigenphi}For $1\leq i\leq N$ and $\S \in\mathsf{Tab}_{\tau}$, $\S \tau (\phi_{i} )=t^{\mathsf{CT}_{\mathbb{S}}[i]}\S$.
\end{Proposition}

\begin{proof}
Arguing by induction, for $i=N$, the result is trivially true since $\mathsf{CT}_{\mathbb{S}}[N]=0$ and $\phi_{N}=1$. Now, suppose that $\S \tau(\phi_{i+1})=t^{\mathsf{CT}_{\mathbb{S}}[i+1]}\S $ for all $\S \in\mathsf{Tab}_{\tau}$. We want to prove that $\S \tau(\phi_{i})=t^{\mathsf{CT}_{\mathbb{S}}[i]}\S $. For that, we study the different cases according to the action formulas of~$\tau(T_{i})$:
\begin{enumerate}\itemsep=0pt
\item[](I)~If $\mathsf{row}_{\mathbb{S}}[i]=\mathsf{row}_{\mathbb{S}}[i+1]$, then $\S \tau(\phi_{i})=\frac{1}{t}\S \tau(T_{i})\tau(\phi_{i+1})\tau(T_{i})=t^{\mathsf{CT}_{\mathbb{S}}[i+1]+1}\S =t^{\mathsf{CT}_{\mathbb{S}}[i]}\S $.

\item[](II)~If $\mathsf{col}_{\mathbb{S}}[i]=\mathsf{col}_{\mathbb{S}}[i+1]$, then $\S \tau(\phi_{i})=\frac{1}{t}\S \tau(T_{i})\tau(\phi_{i+1})\tau(T_{i})=\frac{1}{t}t^{\mathsf{CT}_{\mathbb{S}}[i+1]}\S =t^{\mathsf{CT}_{\mathbb{S}}[i]}\S $.

\item[](III)--(IV)~We compute these two cases at the same time. Let $\mathcal{T}$ and $\Phi$ be the matrices of~$\tau(T_{i})$ and~$\tau(\phi_{i+1})$ respectively, with respect to the basis $\big[\S ,\S^{(i)}\big]$. That is,
\begin{gather*}
\mathcal{T}=
\begin{bmatrix}
-\dfrac{1-t}{1-\varrho} & 1\vspace{1mm}\\
\dfrac{(1-\varrho t)(t-\varrho)}{(1-\varrho)^{2}} & \dfrac{\varrho(
1-t)}{1-\varrho}\end{bmatrix}
, \qquad \Phi=
\begin{bmatrix}
t^{\mathsf{CT}_{\mathbb{S}}[i+1]} & 0\\
0 & t^{\mathsf{CT}_{\mathbb{S}}[i]}
\end{bmatrix},
\end{gather*}
where $\varrho=t^{\mathsf{CT}_{\mathbb{S}}[i+1] - \mathsf{CT}_{\mathbb{S}}[i]}$. A simple calculation shows that $\frac{1}{t}\mathcal{T}\Phi \mathcal{T}=
\left[\begin{smallmatrix}
t^{\mathsf{CT}_{\mathbb{S}}[i]} & 0\\
0 & t^{\mathsf{CT}_{\mathbb{S}}[i+1]}
\end{smallmatrix}\right]$.\hfill \qed
\end{enumerate}\renewcommand{\qed}{}
\end{proof}

The Hecke algebra $\mathcal{H}_{N}(t)$ also acts on polynomials. Let us denote by $\mathcal{P}$ the ring of polynomials $\mathbb{K}[x_{1},\dots,x_{N}]$, where $\mathbb{K}=\mathbb{Q}(t)$ (or $\mathbb{Q}(t,q)$ later on). We denote by~$x$ the set of variables $\{x_{1},\dots,x_{N}\}$ and, for a composition~$\alpha$, $x^{\alpha}=\prod\limits_{i=1}^{N}x_{i}^{\alpha_{i}}$ is a monomial of degree $|\alpha|$. The ring of polynomials $\mathcal{P}$ is graded and we denote by $\mathcal{P}_{n}$ the component of homogeneous polynomials of degree $n\geq0$, i.e., $\mathcal{P}_{n}$ is the span over $\mathbb{K}$ of the monomials~$x^{\alpha}$, for $\alpha$ a composition of~$n$.

We first describe the action of the transposition $s_{i}=(i,i+1)$, for $1\leq i \leq N-1$. For a~composition $\alpha$, $\alpha s_{i}$ is the composition obtained by exchanging $\alpha_{i}$ and $\alpha_{i+1}$. For a polynomial $p\in\mathcal{P}$, $p(x)s_{i} = p(xs_{i})$, that is the polynomial obtained by exchanging $x_{i}$ and $x_{i+1}$. Finally, for $1\leq i \leq N-1$, the operator $T_{i}$ acts on $p\in\mathcal{P}$ by
\begin{gather*}
p(x)T_{i}=(1-t)x_{i+1}\frac{p(x)-p (xs_{i} )}{x_{i}-x_{i+1}}+tp (xs_{i} ).
\end{gather*}
It can be shown straightforwardly that these operators satisfy the defining relations of $\mathcal{H}_{N}(t)$. Moreover, $ps_{i}=p$ if and only if $pT_{i}=tp$, and $pT_{i}=-p$ if and only if $p(x)= (tx_{i}-x_{i+1})p_{0}(x)$, where $p_{0}\in\mathcal{P}$ satisfies $p_{0}s_{i}=p_{0}$.

\begin{remark*}Note we are using the notation $T_{i}$ in some different ways. On one side, there is the abstract $T_{i}$, generator of $\mathcal{H}_{N}(t)$, for which $\tau(T_{i})$ denotes the representation as an operator on a~finite-dimensional vector space, for a given partition $\tau$. On the other side, $T_{i}$ also denotes an operator on the infinite-dimensional space of polynomials. Technically, we should denote it like~$\rho(T_{i})$ since this is another representation of $\mathcal{H}_{N}(t)$. However, one uses $T_{i}$ in both cases since the meaning is clear from the context.
\end{remark*}

Each space $\mathcal{P}_{n}$ can be completely decomposed into subspaces irreducible and invariant under the action of $T_{i}$ in $\mathcal{H}_{N}(t)$. These subspaces have bases of $\{\phi_{i}\}$-simultaneous eigenvectors (or even made up of Macdonald polynomials). Since this is one of the key points of this paper, we introduce the following concept.

\begin{Definition}
A basis $\{p_{\S }\colon \S \in\mathsf{Tab}_{\tau}\}$ of an invariant subspace of $\mathcal{P}_{n}$ is called a \emph{basis of isotype~$\tau$} if each $p_{\S }$ transforms under the action formulas for $T_{i}$ instead of $\tau(T_{i})$ (i.e., repla\-cing~$\tau(T_{i})$ by~$T_{i}$ in the action formulas).
\end{Definition}

The next result is a consequence of Proposition~\ref{eigenphi}.
\begin{Corollary}Let $\{ g_{\S }\colon \S \in\mathsf{Tab}_{\tau}\}$ be a set of polynomials that transforms under the formula actions of $\{T_{i}\}$. Then, $g_{\S }\phi_{i}=t^{\mathsf{CT}_{\mathbb{S}}[i]}g_{\S }$, for all~$\S $ and all~$i$.
\end{Corollary}

The key point here is to figure out when a subspace can have a basis of isotype~$\tau$ made up of Macdonald polynomials, which we introduce in next section.

\subsection{Nonsymmetric Macdonald polynomials and singular polynomials}

In the literature, the different versions of the Macdonald polynomials are usually defined over the double affine Hecke algebra $\mathcal{H}_{N}(q,t)$, where $q$ and $t$ are parameters. For our purpose, it is enough to consider the Hecke algebra $\mathcal{H}_{N}(t)$ together with an extra parameter~$q$. Therefore, we work over the field $\mathbb{K}=\mathbb{Q}(q,t)$. Note that the action and representations defined in Section~\ref{SubSect:Heckealgebra} do not involve~$q$, and keep the same. Moreover, we focus our attention on the nonsymmetric Macdonald polynomials. First, we recall three families of operators~\cite{BF1997,DL2015}.

Given $p\in\mathcal{P}$, the \emph{shift operator} is defined as
\begin{gather*}
p\pi(x)=p (qx_{N},x_{1},x_{2},\ldots,x_{N-1} ).
\end{gather*}
This operator is commonly denoted by $\omega$, but we reserve that notation for the roots of unity that appear later on the paper. The \emph{Cherednik operators} are defined, for $1\leq i\leq N$, as
\begin{gather*}
\xi_{i} = t^{i-1}T_{i-1}^{-1}T_{i-2}^{-1}\cdots T_{1}^{-1}\pi T_{N-1}T_{N-2}\cdots T_{i},
\end{gather*}
where the operator $T_i^{-1} = \frac{1}{t}(T_i+1-t)$ is obtained from $T_i$.
Note that $\xi_{i}=\frac{1}{t}T_{i}\xi_{i+1}T_{i}$ and that the operators $\xi_{i}$ commute with each other. Finally, the \emph{Dunkl operators} are defined recursively by $\mathcal{D}_{N} = \frac{1}{x_{N}}(1-\xi_{N})$, and for $1\leq i \leq N-1$,
$\mathcal{D}_{i}= \frac{1}{t}T_{i}\mathcal{D}_{i+1}T_{i}$. It is a nontrivial but very useful result that $D_{i}$ maps $\mathcal{P}_{n}$ to $\mathcal{P}_{n-1}$.

For a composition $\alpha$, the \emph{nonsymmetric Macdonald polynomials} $M_{\alpha}$ are defined as the basis of simultaneous eigenfunctions for the Cherednik operators with
$\vartriangleright$-leading term $q^{\ast}t^{\ast}x^{\alpha}$, where~$q^{\ast}t^{\ast}$ denotes integer powers of $q$ and $t$, not necessarily the same. That is, for $1\leq i\leq N$,
\begin{gather*}
M_{\alpha}\xi_{i} = q^{\alpha_{i}}t^{N-r_{\alpha}(i)}M_{\alpha},
\end{gather*}
where the eigenvalues $\zeta_{\alpha}(i)=q^{\alpha_{i}}t^{N-r_{\alpha}(i)}$ form the \emph{spectral vector} $\zeta_\alpha= [\zeta_\alpha(1),\dots, \zeta_\alpha(N)]$.

The following result presents two relations that will be very
useful in our study.
\begin{Proposition}[\cite{DL2012}]\label{PropMsiRels}
Let $\rho_{i}=\frac{\zeta_{\alpha}(i+1)}{\zeta_{\alpha}(i)} =q^{\alpha_{i+1}-\alpha_{i}}t^{r_{\alpha}(i)-r_{\alpha}(i+1)}$. Then, for $1\leq i\leq N-1$,
\begin{itemize}\itemsep=0pt
\item if $\alpha_{i}<\alpha_{i+1}$, $\zeta_{\alpha s_{i}}=(\zeta_{\alpha})s_{i}$ and
\begin{gather*}
M_{\alpha}T_{i} = M_{\alpha s_{i}}-\frac{1-t}{1-\rho_{i}}M_{\alpha},\\
M_{\alpha s_{i}}T_{i} = \frac{(1-\rho_{i}t)(t-\rho_{i})}{(1-\rho_{i})^{2}}M_{\alpha}+ \frac{\rho_{i}(1-t)}{(1-\rho_{i})}M_{\alpha s_{i}},
\end{gather*}

\item if $\alpha_{i}=\alpha_{i+1}$, then $M_{\alpha}T_{i}=tM_{\alpha}$.
\end{itemize}
\end{Proposition}

The next result presents an expansion of the nonsymmetric Macdonald polynomial emphasizing its leading term.
\begin{Proposition}[\cite{DL2012}] \label{Prop:Acoefficients}
The nonsymmetric Macdonald polynomials are of the form:
\[
M_{\alpha}(x)=q^{\ast}t^{\ast}x^{\alpha}+ \sum_{\alpha \vartriangleright\beta}A_{\alpha,\beta}(q,t)x^{\beta},
\]
where the coefficients $A_{\alpha,\beta}(q,t)$ are rational functions of~$q$ and~$t$ and whose denominators are of the form $\big(1-q^{a}t^{b}\big)$.
\end{Proposition}

\begin{remark*}Usually in the literature, the nonsymmetric Macdonald polynomials are normalized so its leading coefficient is~1. In our presentation, we consider the construction of the nonsymmetric Macdonald polynomials that uses the Yang--Baxter graph~\cite{DL2011}, and so the multiples of $t$ and $q$ in the leading coefficient come from the raising operator.
\end{remark*}
We say that the parameters $(q,t)$ are \emph{generic parameters} if $q\neq1$ and $q^{a}t^{b}\neq1$, for $a,b\in\mathbb{Z}$ with $|b|\leq N$ and $|a|+|b|>0$.

Singular polynomials appear as a tool used to construct projection maps for vector-valued Macdonald polynomials and to find factorizations connected with highest weight symmetric polynomials~\cite{CDL2019}. In the most general setting, a polynomial $p\in\mathcal{P}$ is said to be \emph{singular} if there exist some specialization of~$(q,t)$ for which $p\xi_{i}=p\phi_{i}$, for all $1\leq i\leq N$. When it comes to nonsymmetric Macdonald polynomials, we have the following equivalent definition.

\begin{Definition}\label{Def:Singular}
A nonsymmetric Macdonald polynomial $M_{\alpha}$ is said to be \emph{singular} for a specific value of $(q,t)$ if the coefficients $A_{\alpha,\beta}(q,t)$ of $M_{\alpha}$ have no poles at $(q,t)$ and $M_{\alpha}\mathcal{D}_{i}=0$, for $1\leq i\leq N$.
\end{Definition}

This formulation of singularity is closely related to the problem of when a shifted (nonhomogeneous) Macdonald polynomial reduces to a homogeneous one (see~\cite[Proposition~2, p.~271]{DL2015}).

The following result shows why the definition for singular polynomials in general coincides with Definition~\ref{Def:Singular}.
\begin{Lemma}Let $p$ be a polynomial and $(q,t)$ be some fixed value. Then $p\xi_{i}=p\phi_{i}$ for all $i$ if and only if $p\mathcal{D}_{i}=0$ for all $i$.
\end{Lemma}
\begin{proof}The definition of $\mathcal{D}_{N}$ shows that $p\mathcal{D}_{N}=0$ if and only if $p\xi_{N}=p=p\phi_{N}$ (since $\phi_{N}=1$). Arguing by induction, suppose the statement holds for $k<i\leq N$. Then,
\begin{align*}
p\mathcal{D}_{k} =0&\iff pT_{k}\mathcal{D}_{k+1}T_{k}=0\iff pT_{k}\mathcal{D}_{k+1}=0 \iff pT_{k}\xi_{k+1}=pT_{k}\phi_{k+1} \\
&\iff pT_{k}\xi_{k+1}T_{k}=pT_{k}\phi_{k+1}T_{k} \iff tp\xi_{k}=tp\phi_{k}.
\end{align*}
This completes the induction.
\end{proof}

We finish this section with an example that illustrates the setup presented.
\begin{example*}Consider the isotype $\tau=(3,1)$ and the special value $qt^{2}=-1$. There are three RSYTs of shape~$\tau$, together with their content and their $\alpha(\S)$-label:
\begin{gather*}
\begin{array}{@{}ccccccc}
& \hspace{1cm}
&\begin{tikzpicture}
\draw (0,0) rectangle (0.5,0.5);
\draw (0,0.5) rectangle (0.5,1);
\draw (0.5,0) rectangle (1,0.5);
\draw (1,0) rectangle (1.5,0.5);
\node at (0.25,0.25) {4};
\node at (0.75,0.25) {3};
\node at (1.25,0.25) {2};
\node at (0.25,0.75) {1};
\end{tikzpicture} & \hspace{2cm} &
\begin{tikzpicture}
\draw (0,0) rectangle (0.5,0.5);
\draw (0,0.5) rectangle (0.5,1);
\draw (0.5,0) rectangle (1,0.5);
\draw (1,0) rectangle (1.5,0.5);
\node at (0.25,0.25) {4};
\node at (0.75,0.25) {3};
\node at (1.25,0.25) {1};
\node at (0.25,0.75) {2};
\end{tikzpicture} & \hspace{2cm} &
\begin{tikzpicture}
\draw (0,0) rectangle (0.5,0.5);
\draw (0,0.5) rectangle (0.5,1);
\draw (0.5,0) rectangle (1,0.5);
\draw (1,0) rectangle (1.5,0.5);
\node at (0.25,0.25) {4};
\node at (0.75,0.25) {2};
\node at (1.25,0.25) {1};
\node at (0.25,0.75) {3};
\end{tikzpicture} \\[0.1in]
\text{Content} & & [-1,2,1,0] & &
[2,-1,1,0] & & [2,1,-1,0] \\[0.07in]
\alpha(\S)\text{-labels} & & (2,0,0,0)
& & (0,2,0,0) & & (0,0,2,0)
\end{array}
\end{gather*}
The spectral vector for $(2,0,0,0)$ is $\big[q^{2}t^{3},t^{2},t,1\big]$, which equals $\big[t^{-1},t^{2},t,1\big]$ when $q^{2}=t^{-4}$. Similar relations hold for $(0,2,0,0)$ and $(0,0,2,0)$. The polynomials $M_{2000}$, $M_{0020}$, and $M_{0020}$ are indeed singular and one need only to show that none of $M_{2000}$, $M_{0200}$, $M_{0020}$, and $M_{0002}$ have poles\footnote{We use the term \emph{poles} to mean the one-dimensional varieties in the $(q,t)$-space $\mathbb{C}^{2}$ defined by the denominators of rational functions of $(q,t)$.} at $qt^{2}=-1$ (an easy computation). Then, $M_{0020}T_{3}=-M_{0020}$ when $qt^{2}=-1$ which follows from the general formula (see Proposition~\ref{PropMsiRels})
\[
M_{0020}T_{3}=\frac{q^{2}t^{3}(1-t)}{1-q^{2}t^{3}}M_{0020}+\frac{t\big(1-q^{2}t^{2}\big) \big(1-q^{2}t^{4}\big) }{\big(1-q^{2}t^{3}\big)^{2}}M_{0002},
\]
where we notice that the coefficients reduce to $-1$ and $0$ when specialized to $q=-t^{-2}$.
\end{example*}

\section{The quasistaircase partitions and the specialization}\label{QUStair}

The nonsymmetric Macdonald polynomials indexed by the quasistaircase and specialized to a~family of parameters are our main object of study in this paper. The quasistaircase partitions can be seen as a generalization of the staircase partitions, which are, in turn, a generalization of the rectangle, which have been studied before. The formula for the specialization of the Jack and Macdonald polynomials in connection with quasistaircases was introduced by Jolicoeur and Luque~\cite{JL2011}. Moreover, it lead to a collaboration between two of the authors of this paper \cite[Section~8]{DL2015}, in which they study the rectangular singular polynomials. Furthermore, it provides another link between nonsymmetric and symmetric polynomials.

In this section, we introduce both the quasistaircase partitions and the specialization, together with useful notation and properties.

The \emph{quasistaircase partition} associated to the parameters $m$, $n$, $d$, $K$, $N$ is the partition
\begin{gather}\label{Eq:quasistaircase}
\lambda=\big(((d+K-1)m)^{\nu_{K}},((d+K-2)m)^{n-1},\dots,(dm)^{n-1} ,0^{dn-1}\big),
\end{gather}
where $\nu_{K}=N-(dn-1)-(K-1)(n-1)$, so that $1\leq\nu_{K}\leq n-1$ and $\lambda$ has $N$ entries in total (including the zero entries).

From now on, $\lambda$ refers to a quasistaircase partitions with the parameters described in~\eqref{Eq:quasistaircase}, unless specified otherwise. We also associate to $\lambda$ two other partitions and a permutation of itself.

\begin{Definition}\label{DefTauNu} Let $\lambda$ be a quasistaircase partition. The \emph{isotype partition} associated to $\lambda$ is the partition defined by $\tau= \big(dn-1, (n-1 )^{K-1},\nu_{K}\big)$, which is a partition of $N$ with length $\ell(\tau)=K+1$. We also define another partition $\nu=(\nu_{0},\nu_{1},\dots,\nu_{K+1})$ recursively by taking $\nu_{1}=N-(dn-1)$, and $\nu_{j+1} = \nu_{j} - (n-1)$, for $1\leq j \leq K-1$. For consistency, we take $\nu_{0}=N$ and $\nu_{K+1} = 0$. Attached to this partition, we consider the intervals of integers given by $I_{j} = [\nu_{j}+1,\nu_{j-1}]$, for $1\leq j \leq K+1$. Intervals are a key object in our study and so, from now on, we denote by $[a,b]$ the interval of integers $[a,b]\cap\mathbb{Z}$.
\end{Definition}
Observe that if $i\in I_{1}$, then $\lambda_{i}=0$, and if $i\in I_{j}$, then $\lambda_{i} = (d+j-2)m$, for $2\leq j\leq K+1$. We also note that $\nu_{a}-\nu_{b}=(n-1)(b-a)$, for $1\leq a,b\leq K$.

\begin{example*}
Consider $\lambda=\big(30^{3},0^{11}\big)$, for which $N=14$, $n=12$, $m=30$, $d=1$, and $K=1$. Therefore, following the definitions above, $\tau=(11,3)$ and $\nu=(14,3,0)$. Moreover, we have two intervals in this case, $I_{1}=[4,14]$ and $I_{2}=[1,3]$.
\end{example*}

\begin{Definition}\label{DefAlpha} For $\S \in\mathsf{RSTab}_{\tau}$, we define a permutation $\alpha(\S) $ of $\lambda$ by setting its entries as
\begin{gather*}
\alpha(\S )_{i}=
\begin{cases}
 (d+\mathsf{row}_{\mathbb{S}}[i]-2 )m, & \text{if }\mathsf{row}_{\mathbb{S}}[i]>1,\\
0, & \text{if } \mathsf{row}_{\mathbb{S}}[i]=1.
\end{cases}
\end{gather*}
\end{Definition}

Note that for $\S_{1}$, $\alpha(\S_1)=\lambda$.
\begin{Lemma}For $\S \in\mathsf{RSTab}_{\tau}$ and $1\leq i\leq N$, the rank function associated to $\alpha(\S )$ is
\begin{gather*}
r_{\alpha(\S )}(i)=
\begin{cases}
\displaystyle \sum_{u=\mathsf{row}_{\mathbb{S}}[i]}^{K+1} \tau_{u} -\mathsf{col}_{\mathbb{S}}[i]+1, & \text{if } \mathsf{row}_{\mathbb{S}}[i]>1,\vspace{1mm}\\
N+1-\mathsf{col}_{\mathbb{S}}[i], & \text{if } \mathsf{row}_{\mathbb{S}}[i]=1.
\end{cases}
\end{gather*}
\end{Lemma}

\begin{proof} If $\mathsf{row}_{\mathbb{S}}[i]=1$, then the entries at positions $(1,\mathsf{col}_{\mathbb{S}}[i])$, $(1,\mathsf{col}_{\mathbb{S}}[i]+1),\dots,(1,nd-1)$ are equal to $\alpha(\S )_{i}$ and the entries in the rest of rows are greater. Thus,
\begin{gather*}
{r_{\alpha(\S )}(i)=nd-1-(\mathsf{col}_{\mathbb{S}}[i]-1)+\sum_{u=2}^{K+1}\tau_{u}=N+1-\mathsf{col}_{\mathbb{S}}[i]}.
\end{gather*}

If $\mathsf{row}_{\mathbb{S}}[i]\geq2$, there are exactly $\tau_{\mathsf{row}_{\mathbb{S}}[i]}-\mathsf{col}_{\mathbb{S}}[i]+1$ parts of $\alpha(\S )$ equal to $\alpha(\S )_{i}$ and $ \sum\limits_{u=\mathsf{row}_{\mathbb{S}}[i]+1}^{K+1} \tau_{u}$ parts that are greater than $\alpha(\S )_{i}$. Therefore, $ r_{\alpha(\S )}(i) =\sum\limits_{u=\mathsf{row}_{\mathbb{S}}[i]}^{K+1}\tau_{u}-\mathsf{col}_{\mathbb{S}}[i]+1$.
\end{proof}

Now that the family of partitions is described, we look at the parameters $q$ and $t$ and specialize them.

\begin{Definition}Consider two integers $m$ and $n$ such that $m\geq1$ and $2\leq n\leq N$. Let $g=\gcd(m,n)$ and $\omega\in\mathbb{C}$ be such that $\omega^{m/g}$ is a primitive $g^{\text{th}}$ root of unity, i.e., $ \omega=\exp\big(\frac{2\pi\mathrm{i}k}{m}\big)$ with $\gcd (k,g )=1$. Define the following \emph{specialization of the parameters} $q$ and~$t$: $\varpi=(q,t)=\big(\omega u^{-n/g},u^{m/g}\big)$ where $u$ is not a root of unity and~$u\neq0$.
\end{Definition}
For the rest of the paper, $F(q,t)|_{\varpi}$ denotes the specialization of $F(q,t)$ in $\varpi$.
Note that $(q,t)=\varpi$ implies $q^{m}t^{n}=1$. In fact, we have the following result.

\begin{Lemma}\label{qmtn}
If there exist integers $a,b$ such that $q^{a}t^{b}\big|_{\varpi} =1$, then there exists $p\in\mathbb{Z}$ such that $a=pm$ and $b=pn$.
\end{Lemma}
\begin{proof}By hypothesis $\omega^{a}u^{-an/g+bm/g}=1$ and, since $u$ is not a root of unity, $-a\frac{n}{g}+b\frac{m}{g}=0$. From $\gcd \big(\frac{n}{g},\frac{m}{g}\big)=1$, it follows that $a=p^{\prime}\frac{m}{g}$ and $b=p^{\prime}\frac{n}{g}$, for some $p^{\prime}\in \mathbb{Z}$.

Thus, $1=\omega^{a}=\exp\big(\frac{2\pi\mathrm{i}k}{m}\frac{mp^{\prime}}{g}\big) =\exp\big(\frac {2\pi\mathrm{i}k}{g}p^{\prime}\big)$. Moreover, since $\gcd (k,g)=1$, $p^{\prime}=pg$ with $p\in\mathbb{Z}$. Hence $a=pm$ and $b=pn$.
\end{proof}

In fact, to describe all the possibilities for $\omega$, it suffices to let $1\leq k<g$. The following result shows that under certain conditions, we can simplify the specialization of $\varpi=(q,t)$.
\begin{Lemma}\label{Lem:possibleomega}
Suppose $\varpi= ( q,t )=\big(\omega u^{-n/g},u^{m/g}\big)$, where $g=\gcd (m,n )$, $u$ is not a root of unity and $u\neq0$, and $\omega=\exp\big(\frac{2\pi\mathrm{i}k}{m}\big)$ with $\gcd (k,g)=1$. Then we can write the factorization as $(q,t)=\big(\exp\big(\frac{2\pi\mathrm{i}k^{\prime}}{m}\big) (u^{\prime} )
^{-n/g}, (u^{\prime} )^{m/g}\big)$, with $\gcd(k^{\prime},g)=1$ and $1\leq k^{\prime}<g$, and $u^\prime$ is not a root of unity again.
\end{Lemma}
\begin{proof}Set $k^{\prime}=k-z_{1}g$ with $z_{1}\in\mathbb{Z}$ such that $1\leq k^{\prime}<g$, that is, $z_{1}=\big\lfloor \frac{k}{g}\big\rfloor $. By definition of $\gcd$, there exist $z_{2},z_{3}\in\mathbb{Z}$ such that $z_{2}m+z_{3}n=g$. Replace $u$ by $\psi u^{\prime}$ where ${\psi=\exp\big(\frac{2\pi\mathrm{i}}{m}z_{1}z_{3}g\big)}$, then $(\psi u^{\prime})^{m/g}=(u^{\prime})^{m/g}$ and
\begin{align*}
\omega u^{-n/g} & =\exp\left(\frac{2\pi\mathrm{i}}{m}(k-nz_{1}z_{3})\right)(u^{\prime})^{-n/g}\\
& =\exp\left(\frac{2\pi\mathrm{i}}{m}(k-z_{1}(g-z_{2}m))\right)(u^{\prime})^{-n/g} =\exp\left(\frac{2\pi\mathrm{i}k^{\prime}}{m}\right) (u^{\prime})^{-n/g}.
\end{align*}
This completes the proof.
\end{proof}

Note that Lemma~\ref{Lem:possibleomega} shows that the number of connected components of the solution set for~$\varpi$ in~$ (\mathbb{C}\backslash \{0 \} )^{2}$ equals $\phi (g )$, where $\phi$ is the Euler function.

Since we study nonsymmetric Macdonald polynomials, the study of the spectral vectors associated is important. The spectral vector for $\alpha(\S)$ has a nice description when specialized.

\begin{Proposition}\label{specvecct} For $1\leq i \leq N$, $\zeta_{\alpha(\S )}(i)|_{\varpi}=t^{\mathsf{CT}_{\mathbb{S}}[i]}$.
\end{Proposition}

\begin{proof}By the definition of the spectral vector, $\zeta_{\alpha(\S )}(i)=q^{\alpha(\S )_{i}}t^{N-r_{\alpha(\S )}(i)}$. Now, we specialize it to $\varpi$. If $\mathsf{row}_{\mathbb{S}}[i]=1$, then $\zeta_{\alpha(\S )}(i)= t^{\mathsf{col}_{\mathbb{S}}[i]-1}=t^{\mathsf{CT}_{\mathbb{S}}[i]}$.
Otherwise, the exponent of $q$ is $(d-2+\mathsf{row}_{\mathbb{S}}[i])m$, and then the exponent of $t$ under the specialization is
\begin{gather*}
N-n(d-2+\mathsf{row}_{\mathbb{S}}[i])- \left(\sum_{u=\mathsf{row}_{\mathbb{S}}[i]}^{K+1}\tau_{u}-\mathsf{col}_{\mathbb{S}}[i]+1\right)\\
\qquad{} =-n(d-2+\mathsf{row}_{\mathbb{S}}[i])+\mathsf{col}_{\mathbb{S}}
[i]-1+(nd-1)+(\mathsf{row}_{\mathbb{S}}[i]-2)(n-1)\\
\qquad{} = \mathsf{col}_{\mathbb{S}}[i]-\mathsf{row}_{\mathbb{S}}[i]=\mathsf{CT}_{\mathbb{S}}[i].\tag*{\qed}
\end{gather*}\renewcommand{\qed}{}
\end{proof}

\section{The equipolar property}\label{Sec4}
The equipolar property appears in this work with the purpose of working with polynomials whose hook length products\footnote{See~\cite{S1999} for more details about \emph{hook length products}.} $h_{q,t}(\alpha,tq)$ vanish at $\varpi$, but for which the poles do not occur when the set of variables is small enough. This property allows us to produce a minimal list of labels~$\alpha$ that have to be analyzed.

\begin{Definition}Let $\alpha$ and $\beta$ be compositions. We say that $M_{\alpha}$ and $M_{\beta}$ are \emph{$\varpi$-equipolar} if $\alpha^{+}=\beta^{+}$ and either both $M_{\alpha}$ and $M_{\beta}$ have no poles at $\varpi$ or both have at least one pole at~$\varpi$.
\end{Definition}

By Proposition~\ref{Prop:Acoefficients}, the coefficient $A_{\alpha,\beta}(q,t)$ is the coefficient of $x^{\beta}$ in $M_{\alpha}$, which is a rational function of $q,t$ whose denominator is of the form $1-q^{a}t^{b}$. Whether $M_{a}$ has a pole at $\varpi$ depends on the presence of a factor $1-q^{mp}t^{np}$, for some integer $p\geq1$, in the denominator. However, the action of $T_i$ by itself introduces no new poles because $x^\gamma T_i$ is a polynomial in~$x$ with coefficients in $\mathbb{Z}[t]$, for any composition $\gamma$. Recall also that $\rho_{i}=\frac{\zeta_{\alpha}(i+1)}{\zeta_{\alpha}(i)}$.

\begin{Proposition}\label{equi1} If $\rho_{i}|_{\varpi}\neq t^{\pm1}$ and $\rho_{i}|_{\varpi}\neq1$, then $M_{\alpha}$ and $M_{\alpha s_{i}}$ are $\varpi$-equipolar.
\end{Proposition}

\begin{proof}Since the relation is symmetric in $\alpha\neq\alpha s_{i}$, we assume that $\alpha_{i}<\alpha_{i+1}$. Moreover, to simplify the notation, we also assume that all the expressions depending on~$q$ and~$t$ appearing in this proof are evaluated at $\varpi$. By the relations described in Proposition~\ref{PropMsiRels},
\begin{gather*}
M_{\alpha s_{i}} =M_{\alpha}T_{i}+\frac{1-t}{1-\rho_{i}}M_{\alpha},\\
M_{\alpha} =\frac{(1-\rho_{i})^{2}}{(1-\rho_{i}t)(t-\rho_{i})}M_{\alpha s_{i}}T_{i}-\frac{\rho_{i} (1-t) (1-\rho_{i})}{(1-\rho_{i}t)(t-\rho_{i})}M_{\alpha s_{i}}.
\end{gather*}
Then, the transformation $M_{\alpha}\rightarrow M_{\alpha s_{i}}$ is invertible for generic parameters $(q,t)$ and introduces no pole at $\varpi$ provided that $\rho_{i}\neq t^{\pm1}$ and $\rho_{i}\neq1$.
\end{proof}

\begin{remark*}
The condition $\rho_{i}|_{\varpi}\neq1$ is necessary for the validity of the proof, even though it is always true for quasistaircases. For instance, for $\alpha
=\big(0,m,1^{n-1}\big)$, $q^{\alpha_{2}-\alpha_{1}}t^{r_{\alpha} (1)-r_{\alpha}(2)}=q^{m}t^{n}$. However, $\alpha$ is not of staircase type.
\end{remark*}

\subsection{Back to Theorem~\ref{MainThm}}\label{SecProofMainThm}
In the introduction we state our main theorem and the goal of this paper. Now, it is time to get back to it. Let us recall it.

\begin{theorem*}[Theorem~\ref{MainThm}]
The polynomials $ \{M_{\alpha(\S)}\colon \S\in\tab_\tau \}$ specialized to $(q,t)=\varpi$ are a~basis of isotype~$\tau$ and are singular.
\end{theorem*}

We have already done part of its proof. First of all, the action formulas for $\tau(T_{i})$ follow from the spectral vector relations described in Proposition~\ref{specvecct}.

By the definition of singular polynomials, Definition~\ref{Def:Singular}, we need to show that for $1\leq i\leq N$, $M_{\alpha (\S )}\xi_{i}=M_{\alpha(\S)}\phi_{i}$. Our idea is to show that no $M_{\alpha(\S)}$ has a pole at $\varpi$ and that if $\mathsf{col}_{\mathbb{S}}[i] =\mathsf{col}_{\mathbb{S}}[i+1]$, for some $i$ and $\S $, then $M_{\alpha(\S)s_{i}}$ has no pole at $\varpi$. This way, we conclude that $M_{\alpha(\S)}T_{i}=-M_{\alpha(\S)}$, and so $M_{\alpha(\S)}\xi_{i}=M_{\alpha(\S)}\phi_{i}$, for $1\leq i\leq N$, by Proposition~\ref{eigenphi}.

These results will take up the rest of the paper. We finish this section with the gist of our approach and how far we are.

Given $\S \in\mathsf{RSTab}_{\tau}$, consider the pair $(\alpha(\S ),\c{S})$. The next two results tell us what happen when $\mathsf{CT}_{\mathbb{S}}[i]-\mathsf{CT}_{\mathbb{S}}[i+1]\geq2$.

\begin{Corollary}\label{eqpolar} Let $\S \in\mathsf{RSTab}_{\tau}$ be such that $\mathsf{CT}_{\mathbb{S}}[i]-\mathsf{CT}_{\mathbb{S}}[i+1]\neq0, \pm1$. Then, $M_{\alpha(\S )}$ and $M_{\alpha(\S s_{i})}$ are $\varpi$-equipolar.
\end{Corollary}

Starting at $\S_{0}$, there is a sequence of steps that end up at $\S $, where each step links $\S^{\prime}$ to $\S^{\prime}s_{i}$ with $\mathsf{row}_{\mathbb{S^{\prime}}}[i] < \mathsf{row}_{\mathbb{S^{\prime}}}[i+1]$ and $\mathsf{col}_{\mathbb{S^{\prime}}}[i]> \mathsf{col}_{\mathbb{S^{\prime}}}[i+1]$. Thus, $\mathsf{CT}_{\mathbb{S^{\prime}}}[i]-\mathsf{CT}_{\mathbb{S^{\prime}}}[i+1]\geq2$ and so, $M_{\alpha
(\S^{\prime})}$ and $M_{\alpha(\S^{\prime}s_{i})}$ are $\varpi$-equipolar.
By an inductive argument on $\mathrm{inv}(\S^{\prime})$, we have the following result.

\begin{Corollary}\label{equi2}
Let $\S \in\mathsf{Tab}_{\tau}$. Then, $M_{\alpha(\S_{0})}$ and
$M_{\alpha(\S )}$ are $\varpi$-equipolar and, equivalently, $M_{\alpha(\S_{1})}$ and $M_{\alpha(\S )}$ are $\varpi$-equipolar.
\end{Corollary}

This means that while there exists some $i$ such that $\mathsf{CT}_{\mathbb{S}}[i]- \mathsf{CT}_{\mathbb{S}}[i+1]\geq2$, we must apply the step $s_{i}$. In this algorithm, the steps $s_{i}$ are under control until no more steps are possible. In the end, the resulting pair $(\alpha(\S^{\prime}),\c{S^{\prime}})$ satisfies that $\mathsf{CT}_{\mathbb{S^{\prime}}}[i]\leq \mathsf{CT}_{\mathbb{S^{\prime}}}[i+1]+1$, for $1\leq i<N$. Therefore, now we have to understand what happens when $\mathsf{CT}_{\mathbb{S}}[i]\leq\mathsf{CT}_{\mathbb{S}}[i+1]+1$.

\section{Critical pairs and the minimal set of configurations}\label{Sec5}

According to the end of the previous section, we are concerned with tableaux with $\mathsf{CT}_{\mathbb{S}}[i+1]=\mathsf{CT}_{\mathbb{S}}[i]+1$, for which Corollary~\ref{eqpolar} do not apply. These tableaux are of the form $\S s_{i}$ where $\S \in\mathsf{Tab}_{\tau}$ and $\mathsf{col}_{\mathbb{S}}[i]=\mathsf{col}_{\mathbb{S}}[i+1]$. The rest of the paper is dedicated to prove the following result.

\begin{Theorem}\label{PropCriticalPairs} Let $\tau$ as in Definition~{\rm \ref{DefTauNu}}. For $\S \in\mathsf{Tab}_{\tau}$ with $\mathsf{col}_{\mathbb{S}}[i]=\mathsf{col}_{\mathbb{S}}[i+1]$ for some $i$, the nonsymmetric Macdonald polynomials $M_{\alpha(\S)}$ and $M_{\alpha (\S s_{i} )}$ in $N$ variables have no poles at~$\varpi$.
\end{Theorem}

Our technique for proving the absence of a pole for a polynomial $M_{\gamma}$ is to show that the spectral vector $\zeta_{\gamma}$ is different from the spectral vector of each element of $ \{ \beta\colon \alpha\vartriangleright \beta,\ell (\beta )\leq N \}$. We use the \textit{critical pair method} to establish this.

Consider two compositions of $N$, $\alpha$ and $\beta$, such that for all $i$, $\zeta_{\alpha}(i)-\zeta_{\beta}(i)|_{\varpi}=0$. This means that
\begin{gather*}
 q^{\alpha_{i}}t^{N-r_{\alpha}(i)}-q^{\beta_{i}}t^{N-r_{\beta}(i)}\big|_{\varpi}= q^{\alpha_{i}}t^{N-r_{\alpha}(i)}\big(1-q^{\beta_{i}- \alpha_{i}}t^{r_{\alpha}(i)-r_{\beta}(i)}\big)\big|_{\varpi}=0.
\end{gather*}
Therefore, by applying Lemma~\ref{qmtn}, there exist integers $p_{i}$ such that $\beta_{i}-\alpha_{i}= mp_{i}$ and $r_{\alpha}(i)-r_{\beta}(i)=np_{i}$, for all~$i$. This motivates the following definition.

\begin{Definition}\label{Def:criticalpairs}
Let $(m,n)\in\mathbb{N}^{2}$ be a pair with $n\geq2$, and take $N^{\prime}\geq N$. We say that the pair of compositions of $N^{\prime}$ $(\alpha,\beta)$ is an \emph{$(m,n)$-critical pair} if $\alpha\vartriangleright\beta$ and there exists $p\in\mathbb{Z}^{N^{\prime}}$ such that $\beta=\alpha+ mp$ and $r_{\alpha}-r_{\beta}=np$.
\end{Definition}

\begin{remark*}Trailing zeros can be adjoined to $\alpha$ and $\beta$ without changing the criticality property. In fact, if $\alpha_{i}=0=\beta_{i}$, for $i\geq i_{0}$, then $r_{a}(i)=i=r_{\beta}(i)$ and $p_{i}=0$. In other words, the definition is independent of $N^{\prime}$ as long as $N^{\prime}$ is sufficiently large. For fixed $\alpha$ and $\beta$, it is enough to take $N^{\prime}\geq\max \{\ell(\alpha),\ell(\beta)\}$. For this paper, $N^{\prime}$ is implicit and large enough unless otherwise is specified.
\end{remark*}

Critical pairs were introduced in~\cite{D2007} by one of the authors of this paper. We use the algorithm included in~\cite{D2007} to produce the second element of the pair when we have the first element of the pair as input. In~\cite{KS1997}, there is a known formula for the least common multiple of the
denominators of the coefficients of $M_{\alpha}$ which involves a certain hook product. However, it assumes that the number of variables is at least $|\alpha|$. Thus, we need a method of handling a restricted number of variables which shows that there is no $\beta$ such that $\ell(\beta)\leq \ell(\alpha)$ and $(\alpha,\beta)$ is a critical pair.

The following is an easy consequence of the definition of critical pairs, Definition~\ref{Def:criticalpairs}.

\begin{Lemma}\label{(m,n)(1,n)}
Let $(\alpha,\beta)$ be a $(m,n)$-critical pair. If $\big(\frac{1}{m}\big)\alpha$ is a composition of $N$, then $\big(\frac{1}{m}\big)\beta$ is a composition of $N^{\prime}$, for some $N^{\prime}\geq N$. Moreover, $\big( \big(\frac{1}{m}\big)\alpha,\big(\frac{1}{m}\big)\beta\big)$ is a $(1,n)$-critical pair. Conversely, if $(\alpha^{\prime},\beta^{\prime})$ is a $(1,n)$-critical pair, then $(m\alpha^{\prime},m\beta^{\prime})$ is a $(m,n)$-critical pair.
\end{Lemma}

We present two other consequences of this definition.

\begin{Lemma}\label{interval} Let $(\alpha,\beta)$ be a $(m,n)$-critical pair. If there exist $i$ and $p$ such that $\alpha_{i}=\alpha_{i+u}$, for $1\leq u\leq p$, and $\beta_{i}=\beta_{i+p}$, then $\beta_{i+u}=\beta_{i}$, for $1\leq u\leq p$.
\end{Lemma}

\begin{proof}Consider the equation $(r_{\beta}(i+p)-r_{\alpha}(i+p))m=n(\alpha_{i}-\beta_{i})$ and subtract from it $(r_{\beta}(i)-r_{\alpha}(i))m=n(\alpha_{i}-\beta_{i})$.
Now, use that $r_{\alpha}(i+p)-r_{\alpha}(i)=p$ to obtain that $r_{\beta}(i+p)-r_{\beta}(i)\allowbreak=p$. Moreover, by the definition of the rank function~\eqref{Def:rankfunct}, $r_{\beta}(i+p)-r_{\beta}(i)=\# \{ u\colon i<u\leq i+p,\beta_{u}=\beta_{i} \}$. Thus, $i<u\leq i+p$ implies that $\beta_{u}=\beta_{i}$.
\end{proof}

\begin{Lemma}\label{Lem:betaizero}Let $(\alpha,\beta)$ be a $(m,n)$-critical pair with $\beta_{i}=0$ for some $i>\ell(\alpha)$. Then, $\beta_{j}=0$, for all $j>i$.
\end{Lemma}

\begin{proof}From $\alpha_{i}=\beta_{i}=0$, it follows that $r_{\beta}(i)=i$. Now, by definition of the rank function~\eqref{Def:rankfunct}, $r_{\beta}(i)=\# \{ j\colon j\leq i,\beta
_{j}\geq0 \}+\# \{ j\colon j>i,\beta_{j}>0 \}$. Thus, $\# \{j\colon j>i,\beta_{j}>0 \}=0$.
\end{proof}

The next result sets up a sufficient condition for having no poles, and that will be used to prove Theorem~\ref{PropCriticalPairs}.

\begin{Proposition}\label{critpoles}
Let $\alpha$ be a composition. Suppose that there is no $\gamma$, with $\ell(\gamma)\leq N$, such that $(\alpha,\gamma)$ is an $(m,n)$-critical pair. Then, $M_{\alpha}$ has no poles at $\varpi$.
That is, the coefficients $A_{\alpha,\beta}(q,t)|_{\varpi}$, with $\alpha\vartriangleright\beta$, are well-defined.
\end{Proposition}

\begin{proof}By the $\vartriangleright$-triangularity of the operators $\xi_{i}$, there are coefficients $b_{\alpha,\beta}(q,t)$ such that
\begin{gather*}
x^{\alpha}=b_{\alpha,\alpha}(q,t) M_{\alpha}+\sum_{\beta\vartriangleleft \alpha}b_{\alpha,\beta}(q,t) M_{\beta},
\end{gather*}
where $b_{\alpha,\alpha}=q^{j}t^{j^{\prime}}$ for some $j,j^{\prime}\in\mathbb{Z}$. For each $\beta\vartriangleleft\alpha$ with $\ell(\beta)\leq N$, there is at least one index~$i[\beta]$ such that $q^{\alpha_{i[\beta]}}t^{N-r_{\alpha} (i[\beta] )}-q^{\beta_{i[\beta]}}t^{N-r_{\beta
}(i[\beta])}\neq0$ at $\varpi$, or else $(\alpha,\beta)$ is a $(m,n)$-critical pair. Define the operator
\begin{gather*}
\mathcal{T}_{\alpha}=\prod\limits_{\beta\vartriangleleft\alpha} \frac{\xi_{i[\beta]}-q^{\beta_{i[\beta]}}t^{N-r_{\beta} (i[\beta] )}}{q^{\alpha_{i[\beta]}}t^{N-r_{\alpha} (i[\beta] )}-q^{\beta_{i[\beta]}}t^{N-r_{\beta} (i[\beta] )}},
\end{gather*}
for which $x^{\alpha}\mathcal{T}_{\alpha}=b_{\alpha,\alpha}(q,t)M_{\alpha}$.
Each factor of $\mathcal{T}_{\alpha}$ maps $M_{\alpha}$ to $M_{\alpha}$ and, for any $\beta\vartriangleleft\alpha$, $M_{\beta}$ is annihilated by at least one factor. Moreover, by construction, the operator $\mathcal{T}_{\alpha}$ has no poles at $\varpi$ and $\prod\limits_{\beta\vartriangleleft\alpha}\big(q^{\alpha_{i[\beta]}}t^{N-r_{\alpha} (i[\beta])}-q^{\beta_{i[\beta]}}t^{N-r_{\beta} (i[\beta])}\big)M_{\alpha}$ has $(q,t)$-polynomial coefficients. Note that none of the terms in the prefactor vanish at $\varpi$.

The formulation shows that $A_{\alpha,\beta}(q,t)$ is a polynomials in $q$ and $t$ divided by a prefactor that does not vanish at $\varpi$, and so it has no poles.
\end{proof}

The rest of the section is dedicated to providing a minimal list of $\S^{\prime}\in\mathsf{RSTab}_{\tau}$, so that $M_{\alpha(\S s_{i})}$ is $\varpi$-equipolar with $M_{\alpha(\S^{\prime})}$. For that, we look at the possible end configurations, starting with~$\S s_{i}$ with~$\S \in\mathsf{Tab}_{\tau}$ and $\mathsf{col}_{\mathbb{S}}[i]=\mathsf{col}_{\mathbb{S}}[i+1]$.

In Definition~\ref{DefTauNu}, we associate two partitions, $\tau$ and $\nu$, to the quasistaircase partition~$\lambda$. We can define $\nu$ in terms of $\tau$ in a more general setting without $\tau$ being the isotype partition of a~quasistaircase partition. Given an arbitrary partition $\tau$ of $N$, we
define a sequence $\nu$ by setting $\nu_{0}=N$ and $\nu_{j} =N-\sum\limits_{i=1}^{j} \tau_{i}$, for $1\leq j\leq\ell(\tau)$. This sequence is related to the inv-minimal RSYT by $\S_{1}[i,1]=\nu_{i-1}$ and $\S_{1}[i,\tau_{i}] = \nu_{i}+1$, for $1\leq i\leq\ell(\tau)$.

\begin{Definition}\label{Def:propertyV}
Let $\S \in\mathsf{RSTab}_{\tau}$. We say that $\S $ has the \emph{property $V(j,k)$}, for some specific~$j$ and $k$, if by interchanging the entries $\S [j,k]$ and $\S [j+1,k]$ we obtain a RSYT. We denote this new RSYT by $\widehat{\S }$ when the values of $j$ and $k$ are clear from the context.
\end{Definition}

That is, except for the entries at $(j,k)$ and $(j+1,k)$, $\S $ agrees with an RSYT. Note that $\S [j,k]>\S [j+1,k]$. It is not necessarily true that performing a vertical interchange on an RSYT leads to such an $\S $, as we can see in the following example.
\begin{example*}
Interchanging the entries with coordinates $(1,2)$ and $(2,2)$ in
$\begin{array}
[c]{|c|c|c|}\hline
5 & 2 & 1\\\hline
6 & 4 & 3\\\hline
\end{array}$
produces
$\begin{array}
[c]{|c|c|c|}\hline
5 & 4 & 1\\\hline
6 & 2 & 3\\\hline
\end{array}$,
which is not in $\mathsf{RSTab}_{\tau}$.
\end{example*}

We need one more definition, in this case, of a particular element among the subset of $\mathsf{RSTab}_{\tau}$ satisfying the property $V(j,k)$.

\begin{Definition}
For $1\leq j<\ell(\tau)$ and $1\leq k\leq\tau_{j+1}$, there exists a distinguished element $\Theta_{j,k}\in\mathsf{RSTab}_{\tau}$ with the property $V(j,k)$. We describe $\Theta_{j,k}$ by rows as follows. For $i\neq j,j+1$, the $i^{\text{th}}$ row of~$\Theta_{j,k}$ agrees with the $i^{\text{th}}$ row of~$\S_{1}$. For $j$ and $j+1$, the corresponding rows of~$\Theta_{j,k}$ are filled with $\nu_{j-1},\nu_{j-1}-1,\ldots,\nu_{j+1}+1$ in a particular way depending on the value of~$k$. We describe them in the following table in which the first row indicated the column index, the second row indicates the
entries in the~$(j+1)^{\text{th}}$ row, and the third row the entries in the~$j^{\text{th}}$ row. In order to make the table more readable, we denote by dots $\cdots$ when we fill with consecutive integers, and we leave empty spots where the entries are zeros.

In general, for $1<k<\tau_{j+1}$,
\begin{gather*}
\begin{array}[c]{|@{\,}cccccccc@{}}
1 & \cdots & k-1 & k & k+1 & \cdots & \tau_{j+1} & \tau_{j}\\\hline & & & & & & & \\
\nu_{j-1}-k+1 & \cdots & \nu_{j-1}-2k+3 & \nu_{j-1}-2k+2 & \nu_{j}-k & \cdots & \nu_{j+1}+1 & \\
\nu_{j-1} & \cdots & \nu_{j-1}-k+2 & \nu_{j-1}-2k+1 & \nu_{j-1}-2k & \cdots & \cdots & \nu_{j}-k+1
\end{array}
\end{gather*}
We also have two special cases. For $k=1$, we just read the table starting from the $k^{\text{th}}$ column. For $k=\tau_{j+1}$, in the $(j+1)^{\text{th}}$ row, all the entries after the entry in the $(\tau_{j+1})^{\text{th}}$ are zero entries.
\end{Definition}

\begin{remark*}The elements $\Theta_{j,k}$ are extremal which means that we
get to the stage when we cannot apply more steps $s_{i}$,
interchanging $i$ and $i+1$, legally in the sense that $\mathsf{row}_{\mathbb{S}}[i]<\mathsf{row}_{\mathbb{S}}[i+1]$ and $\mathsf{col}_{\mathbb{S}}[i]>\mathsf{col}_{\mathbb{S}}[i+1]$.
\end{remark*}

Let us see an example.
\begin{example*}Consider the tableau $\S $ of shape $\tau=\big(4^{3}\big)$ and described below on the left. Then, we can consider its extremal element for $j=k=2$, $\Theta_{2,2}$, which has the property $V(2,2)$, and that we include on the right.
\begin{gather*}
\S =
\begin{array}
[c]{|c|c|c|c|}\hline
8 & 7 & 2 & 1\\\hline
11 & 6 & 5 & 3\\\hline
12 & 10 & 9 & 4\\\hline
\end{array}
\qquad \Theta_{2,2} =
\begin{array}
[c]{|c|c|c|c|}\hline
7 & 6 & 2 & 1\\\hline
8 & 5 & 4 & 3\\\hline
12 & 11 & 10 & 9\\\hline
\end{array}
\end{gather*}
\end{example*}

Our first result claims that $(i,i+1)$ can be interchanged in $\S $ preserving the property $V(j,k)$ provided that $\mathsf{row}_{\mathbb{S}}[i]<\mathsf{row}_{\mathbb{S}}[i+1]$ and that at least one of the rows is not the $j^{\text{th}}$ or the $(j+1)^{\text{th}}$ row.

\begin{Lemma}\label{stepjk} Let $\S $ be a reverse row-ordered standard tableau that has the property $V(j,k)$ and such that $\mathsf{row}_{\mathbb{S}}[i]<\mathsf{row}_{\mathbb{S}}[i+1]$ and $\{\mathsf{row}_{\mathbb{S}}[i],\mathsf{row}_{\mathbb{S}}[i+1]\}\neq\{j,j+1\}$. Then, $\S^{(i)}= \S s_{i}$ also has the property $V(j,k)$.
\end{Lemma}

\begin{proof}The argument has several cases, each more or less obvious.
These cases can be briefly described by $\mathsf{row}_\S[i+1]<j$; $\mathsf{row}_\S[i]>j+1$; $\mathsf{row}_\S[i]<j$ and $\mathsf{row}_\S[i+1]\geq j$; or $\mathsf{row}_\S[i]\leq j+1$ and $\mathsf{row}_\S[i+1]>j+1$.

We prove the case when $\S [j+1,k]=i+1$, and leave the other cases for the reader.

By hypothesis $\mathsf{row}_{\mathbb{S}}[i]<j$ and $\widehat{\S }[j,k]=i+1>\widehat{\S }[j+1,k]$. This implies that $\widehat{\S }[j+1,k]<i$ and $\S s_{i}[j+1,k]=i> \S s_{i}[j,k]$. Also, $\mathsf{col}_{\mathbb{S}}[i]=\mathsf{col}_{\mathbb{\widehat{S}}}[i]<\mathsf{col}_{\mathbb{\widehat{S}}}[i+1] =\mathsf{col}_{\mathbb{S}}[i+1]$. Thus, $\S^{(i)}$ has the property $V(j,k)$.
\end{proof}

Next, we consider the possible transformations of the rows of $\S $ with property $V(j,k)$ other than $j^{\text{th}}$ and $(j+1)^{\text{th}}$ rows.

\begin{Proposition}\label{config1} Let $\S \in\mathsf{Tab}_{\tau}$ be such that $\mathsf{col}_{\mathbb{S}}[u]=\mathsf{col}_{\mathbb{S}}[u+1]=k$, for some $u$, and set $j=\mathsf{row}_{\mathbb{S}}[u+1]$. Then, $\S s_{u}$ has the property $V(j,k)$ and there is a series of steps as in Lemma~{\rm \ref{stepjk}} so that $\S s_{u}$ is transformed to $\S^{\prime}$, where $\S^{\prime}$ agrees with $\S_{1}$ except in the $j^{\text{th}}$ and $(j+1)^{\text{th}}$ rows.
\end{Proposition}

\begin{proof}We proceed by rows, starting with the $1^{\rm st}$ row, unless $j=1$. Suppose the process has arrived at $\S^{\prime}$ with $\S^{\prime}[a,b]=\S_{1}[a,b]$, for $1\leq a<a_{0}<j$ and $1\leq b\leq\tau_{a}$, and for $a=a_{0}$ and $b<b_{0}\leq\tau_{a_{0}}$, with possibly $b_{0}=0$. Then, $v=\S^{\prime}[a_{0},b_{0}]<\S_{1}[a_{0},b_{0}]$ and the entry $v+1$ in~$\S^{\prime}$
must satisfy $\mathsf{row}_{\mathbb{S^{\prime}}}[v+1] > \mathsf{row}_{\mathbb{S^{\prime}}}[v]$ and $\mathsf{col}_{\mathbb{S^{\prime}}}[v+1] < \mathsf{col}_{\mathbb{S^{\prime}}}[v]$. Applying Lemma~\ref{stepjk}, $\S^{\prime}s_{v}$ has the property $V(j,k)$. Continuing in this way leads to $\S^{\prime\prime}$ which agrees with $\S_{1}$ in rows with index~$<j$, and every entry in rows with index $\geq j$ is less than $\nu_{j-1}+1$. Let $z$ be the largest entry in~$j^{\text{th}}$ and~$(j+1)^{\text{th}}$ rows, which is an entry with row index $>j+1$ in $\S_{1}$ and satisfies
\begin{gather*}
z=\max\big(\big( \{\S^{\prime\prime}[j,b]| 1\leq b\leq\tau_{j} \} \cup \{\S^{\prime\prime}[j+1,b^{\prime}]|1\leq b^{\prime}\leq\tau_{j+1} \}\big)\cap[1,\nu_{j+1}]\big).
\end{gather*}
If the intersection is empty, then this part of the process is done. Otherwise $\mathsf{row}_{\mathbb{S^{\prime\prime}}}[z+1]>j+1$ and $\mathsf{row}_{\mathbb{S^{\prime\prime}}}[z]\leq j+1$. Applying once more Lemma~\ref{stepjk}, $\S^{\prime\prime} s_{z}$ has the property $V(j,k)$ and the maximum is increased by $1$, one step closer to the upper limit $\nu_{j+1}=\S_{1}[j+2,1]$.

If the entries in $[\nu_{j+1}+1,\nu_{j-1}]$ are in the $j^{\text{th}}$ and $(j+1)^{\text{th}}$ rows of $\S^{\prime\prime}$, then the process is done.
Otherwise, one of these values is replaced by $\nu_{j+1}$. Let $y$ be the replaced entry, i.e. $y=\S^{\prime\prime}[a,b]$, for some $a>j+1$. If $y=\nu_{j+1}+1$, then $\S^{\prime\prime} s_{y-1}$ has $\nu_{j+1}$ moved to a row with index $>j+1$. Otherwise, $\mathsf{row}_{\mathbb{S^{\prime\prime}}}[y-1]=j$ or $j+1$, and $\S^{\prime\prime} s_{y-1}$ replaces $y$ by $y-1$ in
a row with index $>j+1$ in $\S^{\prime\prime}$. Repeat this process until $y=\nu_{j+1}+1$.
\end{proof}

This proof describes a process for the $1^{\rm st}$ row. We apply it now to all the rows after the $(j+1)^{\text{th}}$ row until these rows agree with the corresponding rows of $\S_{1}$. Once this is done, we describe the values appearing in the $j^{\text{th}}$ and $(j+1)^{\text{th}}$ rows.

\begin{Proposition}\label{config2}
Let $\S \in\mathsf{RSTab}_{\tau}$ such that it has the
property $V(j,k)$ and each row of $\S $ except the $j^{\text{th}}$ and $(j+1)^{\text{th}}$ rows agrees with the corresponding rows of $\S_{1}$. Then,
\begin{gather*}
\S [j,k] =\nu_{j-1}+1-2k,\\
\S [j+1,k] =\nu_{j-1}+2-2k,\\
\bigcup_{1\leq s<k} \{ \S [j,s],\S [j+1,s] \} = [\nu_{j-1}-2k+3,\nu_{j-1} ],\\
\bigcup_{s>k} \{ \S [j,s],\S [j+1,s] \} = [\nu _{j+1}+1,\nu_{j-1}-2k].
\end{gather*}
Furthermore, if we consider the subtableaux of $\ \S $ given by $\{\S [u,v ]\colon j\leq u\leq j+1,1\leq v<k\}$ and $\{\S [u,v]\colon j\leq u\leq j+1,k<v\leq\tau_{u}\}$, we observe that their entries can be arranged to be in row-by-row order, so that the property $V(j,k)$ is preserved in each step and the resulting tableau is~$\Theta_{j,k}$.
\end{Proposition}

\begin{proof}By hypothesis, the entries in the $j^{\text{th}}$ and $(j+1)^{\text{th}}$ rows of $\S $ comprise the interval $[\nu_{j+1}+1,\nu_{j-1}]$. Let $m_{1}=\S [j,k]$ and $m_{2}=\S [j+1,k]$. Then, $m_{1}<m_{2}$ and by row-strictness, $\S [j,b]>m_{2}$ and $\S [j+1,b]>m_{2}$, for $1\leq b<k$.
Observe that by the property $V(j,k)$, Definition~\ref{Def:propertyV}, the tableau with $m_{1}$ and $m_{2}$
interchanged is an RSYT.

Similarly, $\S [j,b]<m_{1}$ and $\S [j+1,b]<m_{1}$, for $b>k$. Thus, the first $2k-2$ entries with column index $<k$, are in the interval $[m_{2}+1,\nu_{j-1}]$. Since the entries of $\S$ are pairwise distinct, it follows that $2k-2\leq\nu_{j-1}-m_{2}$. Analogously, the $\tau_{j}+\tau_{j+1}-2k$ entries of $\S $ in columns with index $>k$ are in the interval $[\nu_{j+1}+1,m_{1}-1]$. Thus,
\begin{gather*}
\tau_{j}+\tau_{j+1}-2k\leq m_{1}-\nu_{j+1}-1=m_{1}- (\nu_{j-1}-\tau_{j}-\tau_{j+1}+1 ).
\end{gather*}
These inequalities imply that $\nu_{j-1}+1-2k\leq m_{1}<m_{2}\leq\nu_{j-1}+2-2k$, and we conclude that $m_{1}=\nu_{j-1}+1-2k$ and $m_{2}=\nu_{j-1}+2-2k$.

This also shows that the first $k-1$ columns form an RSYT with entries $\nu_{j-1}+3-2k\cdots\nu_{j-1}$ and can be transformed to row-by-row order. In the same way, the last $\tau_{j}-k$ columns form an RSYT with entries $\nu_{j+1}+1\cdots\nu_{j-1}+1-2k$.
\end{proof}

\begin{example*}
Consider $\S = \begin{ytableau}
10 & 8 & 7 & 6 & 3 \\
12 & 11 & 9 & 5 & 4 & 2 & 1
\end{ytableau}$, which has the property $V(1,4)$. Then, the row-by-row rearrangement of type $\Theta_{1,4}$ is given by
$\begin{ytableau}
9 & 8 & 7 & 6 & 1 \\
12 & 11 & 10 & 5 & 4 & 3 & 2
\end{ytableau}$.
\end{example*}

\section{Critical pairs for the quasistaircase partitions}
\label{Sec6}

This section includes a series of technical results that lead us to finish our study.

Let $\S \in\mathsf{RSTab}_{\tau}$ with the property $V(j,k)$, for some $j$ and $k$. Applying Corollary~\ref{eqpolar} and Propositions~\ref{config1} and \ref{config2}, $M_{\alpha(\S )}$ and $M_{\alpha (\Theta_{j,k})}$ are $\varpi$-equipolar. For $\Theta_{j,k}$, $\alpha(\Theta_{j,k})$ is defined as follows:
\begin{enumerate}\itemsep=0pt
\item[1)] if $i\leq\nu_{j+1}$ or $i>\nu_{j-1}$ (or equivalently, $\mathsf{row}_{\Theta_{j,k}}[i] \neq j, j+1$), then $\alpha(\Theta_{j,k})_{i}=\lambda_{i}$,
\item[2)] if $\nu_{j+1}+1\leq i\leq\nu_{j}-k$ or $\nu_{j-1}-2k+2\leq i\leq \nu_{j-1}-k+1$, then $\alpha(\Theta_{j,k})_{i}=m(d+j-1)$,

\item[3)] if $\nu_{j}-k+1\leq i\leq\nu_{j-1}-2k+1$ or $\nu_{j-1}-k+2\leq i\leq \nu_{j-1}$ then $\alpha(\Theta_{j,k})_{i}=m(d+j-2)$ for $j>1$, and $\alpha (\Theta_{j,k})_{i}=0$ for $j=1$.
\end{enumerate}

Applying Lemma~\ref{(m,n)(1,n)}, we can assume that $m=1$ in $\alpha (\Theta_{j,k})$, and we denote the resulting composition by~$\mu$.

Let us see an example.
\begin{example*}
Consider $\lambda=\big(4,4,3,3,3,2,2,2,0^{7}\big) $, the quasistaircase with $n=4$, $d=2$, $m=1$. Suppose we apply $s_{9}$ to $\S_{0}$, then $\S_{0} s_{9}$ has property $V(2,2)$. Then,
\begin{gather*}
\S_{0} s_{9}=
\begin{ytableau}
12 & 8 \\
13 & 10 & 5 \\
14 & 9 & 6 \\
15 & 11 & 7 & 4 & 3 & 2 & 1
\end{ytableau} \qquad
\Theta_{2,2}=
\begin{ytableau}
2 & 1 \\
7 & 6 & 3 \\
8 & 5 & 4 \\
15 & 14 & 13 & 12 & 11 & 10 & 9\\
\end{ytableau}
\end{gather*}
and $\mu=\alpha (\Theta_{2,2} )= \big(4,4,3,2,2,3,3,2,0^{7}\big)$. Observe that the location of the two out-of-order entries, $[2,2]$ and $[3,2]$, stays the same. We will show that $ (\alpha (\Theta_{2,2} ),\beta )$ is the only $(1,4)$-critical pair where $\beta=\big(4,4,3,0,0,0,0,3,1^{9}\big)$. Note that $\ell (
\beta )=17=15+2$.
\end{example*}

Now, we want to present an equivalent characterization of the critical pairs, for which we need the following definition.

\begin{Definition} Given a composition $\alpha$, we define the sequence $R_{\alpha}$ by setting $R_{\alpha}(i)=r_{\alpha}(i)\allowbreak +n\alpha_{i}$, for $1\leq i\leq\ell(\alpha)$, and $R_{\alpha}(i)=i$, for $i>\ell(\alpha)$.
\end{Definition}

We use this definition to give another characterization of the critical pairs.

\begin{Lemma}The pair $(\alpha,\beta)$ is a $(1,n)$-critical pair if and only if $\alpha\vartriangleright\beta$ and $R_{\alpha}(i)=R_{\beta}(i)$, for all $i\geq1$.
\end{Lemma}

Our goal in this section is to analyze the $(1,n)$-critical pairs of the form $(\mu, \beta)$. For that, consider a composition $\beta$ such that $\mu\trianglerighteq\beta$ and $R_{\mu}=R_{\beta}$. We refer to these two assumptions as \emph{usual hypothesis}. We assume them for $\beta$ with
respect to $\mu$, but we occasionally replace $\mu$ by $\lambda$.

Once we analyze the $(1,n)$-critical pairs $(\mu,\beta)$, we show that there are no $(1,n)$-critical pairs of the form $(\lambda,\beta)$. This allows us to conclude that $M_{\lambda}$ has no poles in $\varpi$ for any number of variables $\geq\ell (\lambda)$. Taking the idea from~\cite{D2005}, our main tool is applying the maximum principle for the cardinality of the sets $ \{i\colon \beta_{i}=c \}$, for all $c\geq0$.

The arguments in this section are complicated and involve case-by-case studies. That is why this section is split into subsections as follows. In Section~\ref{SetB}, we define the set~$B$, to which we will apply the maximum principle, together with some notation. We also include some useful properties. In Section~\ref{ConseqLambda}, we describe the consequences for $\lambda$ of assuming that $\beta$ satisfies the usual hypothesis with respect to $\lambda$. The last two sections, Sections~\ref{Pairsjlarger1} and~\ref{Pairsj=1}, study the pair~$(\mu,\beta)$ for $j>1$ and $j=1$, respectively. For that, we will do an analysis in terms of different intervals, so we can estimate the size of the set $B$ and the implications
about the possible $\beta$.

\subsection[The set $B$]{The set $\boldsymbol{B}$}\label{SetB}

For $c\geq0$, let $B_{c}=\{i\colon \beta_{i}=c,\, 1\leq i\leq N\}$. In order not to overload the notation, we include $c$ as a subindex of $B$ only when is not clear from the context. Moreover, for intervals $[p_{1},p_{2}]$ and $[p_{3},p_{4}]$, we say $[p_{1},p_{2}]\ll[p_{3},p_{4}]$ if $p_{2}<p_{3}-1$. In
particular, this implies that $\#[p_{1},p_{2}] + \#[p_{3},p_{4}] < \#[p_{1},p_{4}]$.

For $u\neq j,j+1$, by Lemma~\ref{interval}, $B\cap I_{u}$ is either empty or an interval that we denote by $[a_{u},b_{u}]$. Define $s_{u}=\nu_{u-1}-b_{u}$ and $t_{u}=\nu_{u-1}-a_{u}$. Therefore, for $u\neq1, K+1$,
$0\leq s_{u}\leq t_{u}\leq n-2$ and $t_{K+1}\leq\nu_{K}-1$ and $t_{1}\leq nd-2$. Moreover, $\#[a_{u},b_{u}]=\#[s_{u},t_{u}]$.

\begin{Proposition}\label{maxp1} Let $u,u+p\neq j,j+1$ be such that $B\cap I_{u+p}$ and $B\cap I_{u}$ are nonempty and $B\cap I_{s}=\varnothing$, for $u<s<u+p$ $($void if $p=1)$. Then $s_{u+p}-t_{u}=p+1$ and $[s_{u},t_{u}]\ll [s_{u+p},t_{u+p} ]$.
\end{Proposition}

\begin{proof}By definition, $R_{\mu}(b_{u+p})=b_{u+p}+n\mu_{b_{u+p}}$ and $R_{\mu}(a_{u})=a_{u}+n\mu_{a_{u}}$, and by hypothesis, $R_{\beta}(a_{u}) =r_{\beta}(a_{u})+nc=1+r_{\beta}(b_{u+p})+nc=R_{\beta}(b_{u+p})+1$. Since $R_{\mu}=R_{\beta}$, $R_{\mu}(a_{u})-R_{\mu}(b_{u+p})=1$ and then, $1=a_{u}-b_{u+p}+n (\mu_{a_{u}}-\mu_{b_{u+p}} ) $. Thus,
\begin{gather*}
 (\nu_{u-1}-t_{u} )- (\nu_{u+p-1}-s_{u+p} )
=1-n (\mu_{a_{u}}-\mu_{b_{u+p}} ).
\end{gather*}
If $u>1$, then $\nu_{u-1}-\nu_{u+p-1}=p(n-1)$ and $\mu_{a_{u}}-\mu_{b_{u+p}}=(d+u-2)-(d+u+p-2)=-p$ which implies $s_{u+p}-t_{u}=p+1$. For $u=1$, $\nu_{0}-\nu_{p}=nd-1+(p-1)(n-1)$ and $\mu_{a_{1}}-\mu_{b_{1+p}}=-(d+p-1)$ and again $s_{1+p}-t_{1}=p+1$. Thus $t_{u}<s_{u+p}-1$ and $[s_{u},t_{u}]\ll[s_{u+p},t_{u+p}]$.
\end{proof}

We use Proposition~\ref{maxp1} to estimate the size of $B$.

\begin{Corollary}\label{maxp2} If $B\cap (I_{j}\cup I_{j+1})=\varnothing$ and $B$ has a nonempty intersection with at least two intervals $I_{u}$, with $u\neq j,j+1$, then $\#B\leq n-2$. If, additionally, $B\cap I_{K+1}\neq\varnothing$ then $\#B\leq\nu_{K}-1$.
\end{Corollary}

\begin{proof}Suppose $B\cap I_{u_{i}}\neq\varnothing$, for $u_{1}>u_{2}>\cdots>u_{p}$, with $p\geq2$. By Proposition~\ref{maxp1}, $ [s_{u_{p}},t_{u_{p}} ] \ll\cdots\ll [s_{u_{2}},t_{u_{2}} ]\ll [s_{u_{1}},t_{u_{1}} ]$, and so $\#B=\sum\limits_{i=1}^{p} (t_{u_{i}}-s_{u_{i}}+1)$. Notice that the intervals are contained in $[0,t_{u_{1}}]$. Furthermore, $t_{u_{1}}\leq n-2$, for $u_{1}<K+1$, or $t_{u_{1}}=\nu_{K}-1$, for $u_{1}=K+1$. The case $u_{1}=1$ is not possible because $p\geq2$. Note also that there is at least one gap, and therefore, $\#B\leq\#[0,n-2]-1=n-2$, for $u_{1}>K+1$, and $\#B\leq\nu_{K}-1$, for $u_{1}=K+1$.
\end{proof}

\begin{remark*}Proposition~\ref{maxp1} and Corollary~\ref{maxp2} apply to $\lambda$ without the exclusion $u,u+p\neq j,j+1$.
\end{remark*}

\subsection[Consequences of the usual hypothesis for $\lambda$]{Consequences of the \emph{usual hypothesis} for $\boldsymbol{\lambda}$}\label{ConseqLambda}

In this section we show that $\lambda\trianglerighteq\beta$ and $R_{\lambda}=R_{\beta}$ imply that $\beta=\lambda$.

\begin{Lemma}\label{LemBeta}
If $R_{\lambda}=R_{\beta}$, then $\beta_{\nu_{1}+1}=0$.
\end{Lemma}

\begin{proof}
On one hand, since $\ell(\lambda)=\nu_{1}$, $R_{\lambda}(\nu_{1}+1)=\nu_{1}+1$. On the other hand, by definition, $R_{\beta}(\nu_{1}+1)=r_{\beta}(\nu_{1}+1)+n\beta_{\nu_{1}+1}$. Setting $\beta_{\nu_{1}+1}=b$, we obtain that
\begin{align*}
r_{\beta}(\nu_{1}+1)& =\nu_{1}+1-n\beta_{\nu_{1}+1}\\
& = \nu_{1}+1-\# \{i\colon 1\leq i\leq\nu_{1},\,\beta_{i}<b \}+\# \{ i \colon i>\nu_{1}+1,\,\beta_{i}>b \} \\
& \geq\nu_{1}+1-\sum_{s=0}^{b-1}\# \{i \colon 1\leq i\leq\nu_{1},\,\beta_{i}=s \}\geq\nu_{1}+1-b(n-1).
\end{align*}
By Corollary~\ref{maxp2}, $\# \{ i \colon 1\leq i\leq\nu_{1},\,\beta_{i}=s \}\leq n-1$ because the bound $i\leq\nu_{1}$ excludes $I_{1}$ and the other intervals satisfy $\#I_{u}\leq n-1$. Thus, $-nb=r_{\beta}(\nu_{1}+1)-\nu_{1}-1\geq-b(n-1)$, and therefore, $b=0$.
\end{proof}
\begin{remark*}Note that by Lemma~\ref{Lem:betaizero} this implies that $\beta_i =0$ for $i> \nu_1$.
\end{remark*}

\begin{Lemma} If $\lambda\trianglerighteq\beta$ and $R_{\lambda}=R_{\beta}$, then $\beta$ is a permutation of $\lambda$.
\end{Lemma}

\begin{proof}Since $\ell(\lambda)=\nu_1$ and $\lambda\trianglerighteq\beta$, we have that $\ell(\beta)\geq\nu_{1}$. However, by Lemma~\ref{LemBeta}, we also have that $\ell(\beta)\leq\nu_{1}$. Therefore, we conclude that $\ell(\beta)=\nu_{1}$.
Moreover, $d\leq\beta_{i}\leq d+K-1$, for $1\leq i\leq\nu_{1}$. For $2\leq i\leq K+1$, let $D_{i}= \{ s\colon \beta_{s}=d+i-2 \}$ and $m_{i}=\#D_{i}-(\nu_{i-1}-\nu_{i})$. By Corollary~\ref{maxp2}, $m_{i}\leq0$.
Also $\sum\limits_{i=2}^{K+1}m_{i}=0$ and thus $m_{i}=0$, for all $i$, and $\beta^{+}=\lambda$.
\end{proof}

We are ready to prove the following result.

\begin{Proposition}\label{beta1} If $\lambda\trianglerighteq\beta$, $R_{\lambda}=R_{\beta}$ and $\beta^{+}=\lambda$, then $\beta=\lambda$.
\end{Proposition}

\begin{proof}By definition, $R_{\lambda}(\nu_{K})=\nu_{K}+n(d+K-1)$, and since $R_{\lambda}=R_{\beta}$, we also have that $R_{\lambda}(\nu_{K})=r_{\beta}(\nu_{K})+n\beta_{\nu_{K}}$. Suppose $\beta_{\nu_{K}}=d+u-2$, with $u<K+1$. Using that $\# \{s \colon \beta_{s}>d+u-2 \}=\# \{s \colon \beta_{s}^{+}>d+u-2 \}=\nu_{u}$, we obtain that
\begin{align*}
r_{\beta}(\nu_{K})& =\# \{s \colon s\leq\nu_{K},\,\beta_{s}=d+u-2 \}+\# \{s \colon \beta_{s}>d+u-2 \} \\
& =\# \{s \colon s\leq\nu_{K},\, \beta_{s}=d+u-2 \}+\nu_{u}\leq\nu_{K}+\nu_{u},
\end{align*}
Putting all together, we get the inequality $\nu_{K}+n(d+K-1)-n(d+u-2)\leq \nu_{K}+\nu_{u}$, which contradicts the condition $\beta_{\nu_{K}} <\lambda_{\nu_{K}}$. Therefore, $\beta_{\nu_{K}}=d+K-1$ and $r_{\beta}(\nu_{K})=\nu_{K}$. The hypothesis $\lambda\trianglerighteq\beta$ implies that
$\beta_{i}\leq d+K-1$ and so, $\beta_{i}=d+K-1$, for $1\leq i\leq\nu_{K}$.

Arguing inductively, suppose $\beta_{i}=\lambda_{i}$, for $i\in \bigcup\limits_{s=u+1}^{K+1}I_{s}$. The possible nonzero values of $\beta$ on~$I_{u}$ are $d+s-2$, for $2\leq s\leq u$. Consider $R_{\beta}(\nu_{u-1})=R_{\lambda}(\nu_{u-1})=\nu_{u-1}+n(d+u-2)$. A similar argument shows that $\beta_{i}=d+u-2$, for $i\in I_{u}$, and therefore $\beta=\lambda$.
\end{proof}

\subsection[The pairs $(\alpha(\Theta_{j,k}),\beta)$: Case $j>1$]{The pairs $\boldsymbol{(\alpha(\Theta_{j,k}),\beta)}$: Case $\boldsymbol{j>1}$}\label{Pairsjlarger1}

In this case, we look at the set $I_{j}\cup I_{j+1}$, with $j>1$, by splitting it into four intervals. These intervals, together with their key properties, are:

\begin{itemize}\itemsep=0pt
\item $E_{1}= [\nu_{j+1}+1,\nu_{j}-k ]$, with $E_{1}=\varnothing$, for $k=\tau_{j+1}$. For $i\in E_{1}$, $\mu_{i}= d+j-1$ and $r_{\mu}(i)=i$. Moreover, $\#E_{1}=\tau_{j+1}-k$.

\item $E_{2}= [\nu_{j}-k+1,\nu_{j-1}-2k+1 ]$. For $i\in E_{2}$, $\mu_{i}=d+j-2$ and $r_{\mu}(i)=i+k$. Moreover, $\#E_{2}=n-k$.

\item $E_{3}= [\nu_{j-1}-2k+2,\nu_{j-1}-k+1 ]$. If $i\in E_{3}$, $\mu_{i}=d+j-1$ and $r_{\mu}(i)=i-n+k$. Moreover, $\#E_{3}=k$.

\item $E_{4}= [\nu_{j-1}-k+2,\nu_{j-1} ]$, with $E_{4} =\varnothing$, for $k=1$. For $i\in E_{4}$, $\mu_{i}=d+j-2$ and $r_{\mu}(i)=i$. Moreover, $\#E_{4}=k-1$.
\end{itemize}

Recall that our goal is to describe the possible compositions $\beta$ such that $R_{\mu}= R_{\beta}$. We claim that there is a unique such $\beta$ and that is of the form:
\begin{gather*}
\beta_{i}=
\begin{cases}
\mu_{i} +1 & \text{for } i\in I_{u}, \text{ with } 1<u<j,\\
d+j-1 & \text{for } i\in E_{1}\cup E_{4},\\
0 & \text{for } i \in E_{2}\cup E_{3},\\
\mu_{i} & \text{for } i\in I_{u}, \text{ with } u>j+1,\\
1 & \text{for } \nu_{1}<i \leq N+j.
\end{cases}
\end{gather*}

To prove that such $\beta$ satisfies $R_{\beta}=R_{\mu}$ it suffices to check a few cases:
\begin{itemize}\itemsep=0pt
\item If $i\leq\nu_{j+1}$ or $i\in E_{1}$, then $\lambda_{i}=\beta_{i}$ and $r_{\mu}(i)=i=r_{\beta}(i)$. Note that if $E_{2}=E_{3}$ is excluded, then $\beta$ is nonincreasing.

\item If $i>\nu_{j-1}$ or $i\in E_{4}$, then $r_{\beta}(i)=i-n$. As a consequence, $R_{\beta}(i)=r_{\beta}(i)+n\beta_{i}= i-n+n(\mu_{i}+1) = R_{\mu}(i)$.

\item If $i=\min E_{2}$, then $R_{\beta}(i)=N+j-n+1$ and $R_{\mu}(i)=\nu_{j}+1+n(d+j-2) =N+j-n+1$.

\item If $i=\max E_{3}$, then $R_{\beta}(i)=N+j$ and $R_{\mu}(i) =\nu_{j}+n(d+j-1)=N+j$.
\end{itemize}

The challenge is to prove the uniqueness of $\beta$.

Our first step is to extend the maximum principle to $B\cap (I_{j}\cup I_{j+1} )$. For that, we describe the analogues of the intervals $[a_{u},b_{u}]$ and $[s_{u},t_{u}]$ for $B\cap E_{i}$, $1\leq i\leq4$.
\begin{itemize}\itemsep=0pt
\item $B\cap E_{1}=[a_{j+1},b_{j+1}]$ and $[s_{j+1},t_{j+1}]=[\nu_{j}-b_{j+1},\nu_{j}-a_{j+1}]$, with $k\leq s_{j+1}\leq t_{j+1}\leq \tau_{j+1}-1$.

\item $B\cap E_{2}=[a_{j}^{\prime}, b_{j}^{\prime}]$ and $[s_{j}^{\prime},t_{j}^{\prime}]=[\nu_{j-1}
-k-b_{j}^{\prime}, \nu_{j-1}-k-a_{j}]$, with $k-1\leq s_{j}^{\prime}\leq t_{j}^{\prime}\leq n-2$.

\item $B\cap E_{3}=[a_{j+1}^{\prime},b_{j+1}^{\prime}]$ and $[s_{j+1}^{\prime},t_{j+1}^{\prime}]= [\nu_{j-1}-k+1-b_{j+1}^{\prime},\nu_{j-1}-k+1-a_{j+1}^{\prime}]$, with $0\leq s_{j+1}^{\prime}\leq t_{j+1}^{\prime}\leq k-1$.

\item $B\cap E_{4}=[a_{j},b_{j}]$ and $[s_{j},t_{j}]=[\nu_{j-1}-b_{j},\nu_{j-1}-a_{j}]$, with $0\leq s_{j}\leq t_{j}\leq k-2$.
\end{itemize}

If $B\cap E_{i}$ is empty for some $i$, the corresponding interval is omitted.
We introduce a \emph{shorthand notation} for the possible states of $B\cap E_{s}$: set $\mathbf{b}=(\mathbf{b}_{i})_{i=1}^{4}$ where $\mathbf{b}_{i}=1$ if $B\cap E_{i}\neq\varnothing$, $\mathbf{b}_{i}=0$ if $B\cap E_{i}=\varnothing$, and $\mathbf{b}_{i}=\ast$ if either is possible.

We list here the consequences of the rank equation $R_{\beta}= R_{\mu}$ according to the possible values of $\mathbf{b}$.
\begin{itemize}\itemsep=0pt
\item For $\mathbf{b}=(11\ast\ast)$, $a_{j}^{\prime}=b_{j+1}+n-k+1$ and $s_{j+1}=t_{j}^{\prime}+2$, so then $[s_{j}^{\prime},t_{j}^{\prime}]\ll[s_{j+1},t_{j+1}]$. Note also that $r_{\mu}(a_{j}^{\prime})=a_{j}^{\prime}+k$.
\item For $\mathbf{b}=(\ast\ast11)$, $a_{j}=b_{j+1}^{\prime}+k+1$ and $s_{j+1}^{\prime}=t_{j}+2$, so then $[s_{j},t_{j}] \ll[s_{j+1}^{\prime}, t_{j+1}^{\prime}]$. Note also that $r_{\mu}(b_{j+1}^{\prime})=b_{j+1}^{\prime}-n+k$.

\item For $\mathbf{b}=(101\ast)$, $a_{j+1}^{\prime}=b_{j+1}+n-k+1$ and $s_{j+1}=t_{j+1}^{\prime}+1$. Moreover, $[s_{j+1}^{\prime},t_{j+1}^{\prime}]$ and $[s_{j+1},t_{j+1}]$ are contiguous and $\#(B\cap(E_{1}\cup E_{3})) =\#[s_{j+1}^{\prime},t_{j+1}]$. Also $b_{j+1}=\nu_{j}-k$ and $a_{j+1}^{\prime}=\nu_{j-1}-2k+2$ since $\#E_{2}=n-k$.

\item For $\mathbf{b}=(\ast101)$, $a_{j}=b_{j}^{\prime}+k+1$ and $s_{j}^{\prime}=t_{j}+1$. Then $[s_{j},t_{j}]$ and $[s_{j}^{\prime},t_{j}^{\prime}]$ are contiguous and $\#(B\cap(E_{2}\cup E_{4}))=\#[s_{j},t_{j}^{\prime}]$. Also $b_{j}^{\prime}=\nu_{j-1}-2k+1$ and $a_{j}=\nu_{j-1}-k+2$.

\item For $\mathbf{b}=(\ast11\ast)$, $a_{j+1}^{\prime}=b_{j}^{\prime}+1$ and $t_{j+1}^{\prime}=s_{j}^{\prime}$. Also $a_{j+1}^{\prime}=\nu_{j-1}-2k+2,b_{j}^{\prime}=\nu_{j-1}-2k+1$ and $s_{j}^{\prime}=k-1$. Thus,
$\#(B\cap(E_{2}\cup E_{3}))=\#[s_{j+1}^{\prime},t_{j}^{\prime}]+1$.
\end{itemize}

We give more detail on those cases with more non-empty intersection.
\begin{itemize}\itemsep=0pt
\item For $\mathbf{b}=(\ast111)$, $s_{j+1}^{\prime}=t_{j}+2$ and $[s_{j},t_{j}]\ll[s_{j+1}^{\prime},t_{j}^{\prime}]$.
Furthermore,
\begin{gather*}
\#B\cap (E_{2}\cup E_{3}\cup E_{4} )= (t_{j}-s_{j}
+1 )+ (t_{j}^{\prime}-s_{j+1}^{\prime}+2 )=t_{j}^{\prime}-s_{j}+1 =\# [s_{j},t_{j}^{\prime} ],
\end{gather*}
giving an upper bound of $n-1$.

\item For $\mathbf{b}=(111\ast)$, $s_{j+1}=t_{j}^{\prime}+2$ and $[s_{j+1}^{\prime},t_{j}^{\prime}]\ll[s_{j+1},t_{j+1}]$.
Moreover,
\begin{gather*}
\#B\cap (E_{1}\cup E_{2}\cup E_{3} )= (t_{j+1}-s_{j+1}+1 )+ (t_{j}^{\prime}-s_{j+1}^{\prime}+2 )\\
\hphantom{\#B\cap (E_{1}\cup E_{2}\cup E_{3} )}{}
=t_{j+1}-s_{j+1}^{\prime}+1=\# [s_{j+1}^{\prime},t_{j+1} ].
\end{gather*}

\item For $\mathbf{b}=(1111)$, $[s_{j},t_{j}]\ll[s_{j+1}^{\prime},t_{j}^{\prime}]\ll[s_{j+1},t_{j+1}]$.
Furthermore,
\begin{gather*}
\#B\cap (E_{1}\cup E_{2}\cup E_{3}\cup E_{4} )= (
t_{j+1}-s_{j+1}+1 )+ (t_{j}^{\prime}-s_{j+1}^{\prime}+2 )+ (t_{j}-s_{j}+1 )\\
\hphantom{\#B\cap (E_{1}\cup E_{2}\cup E_{3}\cup E_{4} )}{} =t_{j+1}-s_{j}=\# [s_{j},t_{j+1} ]-1.
\end{gather*}
\end{itemize}

The next three results give us an estimate for the size of $B$, obtained by studying the cases depending on its intersection with the intervals $E_{i}$. The bounds for $\#B$ are analyzed by arguments depending on which of the intersections of~$B$ with $E_{1}$, $E_{2}$, $E_{3}$, $E_{4}$ are non-empty. Some of the 16 possibilities can be combined for this purpose. Our first result cover the case when $B\cap(E_{2}\cup E_{3})=\varnothing$ and it can be obtained reproducing the proof for Proposition~\ref{maxp1} and Corollary~\ref{maxp2}. Note that the case $B\subset E_{1}\cup E_{4}$ implies the states $(100\ast)$ or $(\ast001)$ and is treated in the following result.
\begin{Corollary} \label{SizeB1} If $\mathbf{b}=(100\ast)$ or $\mathbf{b}=(\ast001)$ and $B\cap I_{u}\neq\varnothing$ for at least one value of $u$, then $\#B\leq n-2$. If, additionally, $B\cap I_{K+1}\neq\varnothing$, then $\#B\leq\nu_{K}-1$.
\end{Corollary}
Observe that if $\mathbf{b}= (1001 )$, then $s_{j+1}=t_{j}+2$ and $[s_{j},t_{j}]\ll [s_{j+1},t_{j+1} ] $. Moreover, the case $B \subset E_1\cup E_4$ implies the states $(100\ast)$ or $(\ast 001)$ and it is treated in Corollary~\ref{SizeB1}.

Now, we look at the case in which $B\cap (E_{2}\cup E_{3} )\neq \varnothing$, splitting it into two cases depending on the intersection of~$B$ with the intervals $I_u$.
\begin{Proposition}\label{SizeB2} If $B\cap(E_{2}\cup E_{3})\neq\varnothing$ and at least one $B\cap I_{u}\neq\varnothing$, for some $u\neq j,j+1$, then $\#B\leq n-2$.
\end{Proposition}

\begin{proof}
First, we consider the case $B\cap I_{u}\neq\varnothing$, with $u>j+1$ such that $B\cap I_{s}=\varnothing$, for $j+1<s<u$. We look at the possible configurations of $\mathbf{b}$ starting from the left. If $\mathbf{b}=(1\ast\ast\ast)$, by Proposition~\ref{maxp1}, $s_{u}-t_{j+1}=u-j\geq2$ and $ [s_{j+1},t_{j+1} ]\ll [s_{u},t_{u} ]$. Furthermore, $\#B\cap (E_{1}\cup E_{2}\cup E_{3}\cup E_{4} )=\# [s_{j+1}^{\prime},t_{j+1} ]$ or $\# [s_{j},t_{j+1} ]-1$, depending on $B\cap E_{4}$. If $\mathbf{b}= (01\ast\ast)$, then the rank equations show that $s_{u}-t_{j}^{\prime}=1+u-j\geq3$ and $[s_{j}^{\prime},t_{j}^{\prime}+1 ]\ll [s_{u},t_{u} ]$. If $\mathbf{b}=(\ast11\ast)$, then $\# (B\cap (E_{2}\cup E_{3} ) )=\# [s_{j+1}^{\prime},t_{j}^{\prime} ]+1$.
Finally, if $\mathbf{b}= (001\ast )$, then $s_{u}\geq t_{j+1}^{\prime}+2$ and $ [s_{j+1}^{\prime},t_{j+1}^{\prime} ]\ll [s_{u},t_{u} ]$.

Now, consider the other case, $B\cap I_{u}\neq\varnothing$, with $u<j$ such that $B\cap I_{s}=\varnothing$, for $j>s>u$. Again, we look at the possible configurations of $\mathbf{b}$ starting now from the right. If $\mathbf{b}= (\ast\ast\ast1 )$, then $s_{j}-t_{u}=j-u+1$, even when $u=1$, and $ [s_{u},t_{u} ]\ll [s_{j},t_{j} ]$. If $\mathbf{b}= (\ast\ast10 )$, then $s_{j+1}^{\prime}-t_{u}=j-u+2\geq3$. Thus, $ [s_{u},t_{u} ]\ll [ s_{j+1}^{\prime}-1,t_{j+1}^{\prime} ]$ and $\# (B\cap (E_{2}\cup E_{3} ) )=\# [s_{j+1}^{\prime}-1,t_{j}^{\prime} ]$ or $\# [s_{j+1}^{\prime},t_{j+1}^{\prime} ]$. Finally if $\mathbf{b}= (\ast100 )$ then $s_{j}^{\prime}-t_{u}=j-u+1\geq2$ and $ [s_{u},t_{u} ]\ll [s_{j}^{\prime},t_{j}^{\prime} ]$.

In all cases, $B$ has the same cardinality as a union of disjoint subintervals of $ [0,n-2 ]$, with gaps of at least one between adjacent subintervals. Thus, $\#B\leq n-2$ and $\leq\nu_{K}-1$, if $B\cap I_{K+1}\neq\varnothing$.
\end{proof}
Now, we consider the case in which $B\cap (E_{2}\cup E_{3})\neq\varnothing$ but the intersection of $B$ with the intervals~$I_u$ is empty.
\begin{Proposition}\label{SizeB3} We list here the exceptional cases, for which $B\cap I_{u}=\varnothing$, for $u\neq j,j+1$.
\begin{itemize}\itemsep=0pt
\item For $\mathbf{b}=(0\ast\ast0)$, $\#B\leq n$ with $\#B=n$ if and only if $B=E_{2}\cup E_{3}$.
\item For $\mathbf{b}=(1110)$, $\#B\leq n-1$.
\item For $\mathbf{b}=(0111)$, $\#B\leq n-1$.
\end{itemize}
\end{Proposition}

Finally, if we are not in any case included in Corollary~\ref{SizeB1} or Propositions~\ref{SizeB2} and~\ref{SizeB3}, $B\subset I_{u}$ for some $u\neq j,j+1$ or $B\subset E_{s}$, with $1\leq s\leq4$.

Our next step is to analyze the implications of these results with respect to the possible compositions $\beta$. First, we notice that since we are assuming $j>1$, we know that for $\mu$, $\ell(\mu)= \nu_{1}=N-(nd-1)$, and so $\mu_{i}=0$, for $i>\nu_{1}$. Next lemma tells us this information for~$\beta$.

\begin{Lemma}\label{Betai}Either $\beta_{i}=0$ for all $i\geq\nu_{1}+1$ or $\beta_{i}=1$
for $\nu_{1}+1\leq i\leq\ell(\beta)$. Moreover, in the last case, $r_{\beta}(\nu_{1}+1)=\nu_{1}+1-n$ and $ \{ i\colon 1\leq i\leq\nu_{1},\beta_{i}=0 \}=E_{2}\cup E_{3}$.
\end{Lemma}

\begin{proof}Let $b=\beta_{\nu_{1}+1}$. By the rank equation, $r_{\beta}(\nu_{1}+1)+nb= (\nu_{1}+1)$, since $R_{\beta}=R_{\mu}$ and $\mu_{\nu_{1}+1}=0$. Then, $\nu_{1}+1-nb=r_{\beta}(\nu_{1}+1)\geq1$. By definition,
\begin{gather*}
r_{\beta}(\nu_{1}+1)=\nu_{1}+1-\#\{ i\colon 1\leq i\leq\nu_{1},\, \beta_{i}<b\}+\#\{ i\colon i>\nu_{1},\, \beta_{i}>b\}.
\end{gather*}
We already know that $\# \{ i\colon 1\leq i\leq\nu_{1},\,\beta_{i}=c \}\leq n-1$, with one possible exception of $n$, in which we have exactly $E_{2}\cup E_{3}$, by Proposition~\ref{SizeB3}. Since we are considering subsets of $[1,\nu_{1}]$, then it is not possible to have $\# \{ i\colon 1\leq
i\leq\nu_{1},\,\beta_{i}=c \}> n$. Moreover, the interval $ [\nu_{1}+1,N]$ is excluded here, so values in $[n+1,nd-1]$ are excluded. Thus,
\begin{gather*}
\# \{ i\colon 1\leq i\leq\nu_{1},\, \beta_{i}<b \}= \sum_{c=0}^{b-1} \#\{i\colon \beta_{i}=c,\, 1\leq i \leq\nu_{1}\}\\
\hphantom{\# \{ i\colon 1\leq i\leq\nu_{1},\, \beta_{i}<b \}}{} \leq(n-1)(b-1)+n = b(n-1)+1.
\end{gather*}
Putting all together,
\begin{gather*}
\nu_{1}+1-nb=r_{\beta} (\nu_{1}+1 )\geq\nu_{1}+1-\# \{i\colon 1\leq i\leq\nu_{1},\,\beta_{i}<b \}\geq\nu_{1}-b(n-1).
\end{gather*}
That is $nb-1\leq b(n-1)$, and so $b\leq1$. If $b=0$, then $r_{\beta} (\nu_{1}+1 )=\nu_{1}+1$ which implies $\# \{ i\colon i>\nu_{1},\,\beta_{i}>0 \}=0$ and $\beta_{i}=0$ for $i>\nu_{1}$. Otherwise, $b=1$ and $r_{\beta} (\nu_{1}+1 )=\nu_{1}+1-n$. According to the notation
described in Section~\ref{SetB}, let $B_{0}= \{ i\colon 1\leq i\leq\nu_{1},\,\beta_{i}=0\}$. By the results about the size of $B$ presented in Section~\ref{SetB}, $\#B_{0}\leq n$ and $\# \{i\colon i>\nu_{1},\, \beta_{i}>1 \}=0$. We conclude then that $\#B_{0}=n$ and $B_{0}=E_{2}\cup E_{3}$. Furthermore,
$\nu_{1}+1\leq i\leq\ell (\beta )$ implies $\beta_{i}=1$ because the values $\beta_{i}>1$ and $\beta_{i}=0$ are excluded.
\end{proof}

In fact, we also know the length of $\beta$ for the last case in Lemma~\ref{Betai} as we show in the following proposition.

\begin{Proposition}\label{LenBeta1}If $\beta_{\nu_{1}+1}=1$, then $\ell(\beta)=N+j$.
\end{Proposition}

\begin{proof}By Lemma~\ref{Betai}, $i_{0}=\min \{ i\colon \beta_{i}=0 \}=\min E_{2}=\nu_{j}-k+1$. Then $r_{\beta}(i_{0})=1+\# \{ i\colon 1\leq i\leq \ell(\beta),\,\beta_{i}>0 \}=\ell(\beta)+1-n$ and the rank equation $R_{\beta}(i_{0})=R_{\mu}(i_{0})$ implies that $n(d+j-2)=r_{\beta}(i_{0})-r_{\mu}(i_{0})=\ell(\beta)+1-n-(i_{0}+k)=\ell(\beta)+1-n-\nu_{j}-1$.
Substitute $\nu_{j}=N-(nd-1)-(j-1)(n-1)$ in the last equation and obtain $\ell(\beta)=N+j$.
\end{proof}

We are ready to prove how is $\beta$ in this last case.

\begin{Theorem}\label{FinalThm1} If $R_{\beta}=R_{\mu}$, $\mu\vartriangleright\beta$, $j>1$
and $\beta_{\nu_{1}+1}=1$, then
\begin{gather*}
\beta_{i} = \begin{cases}
\mu_{i} & \text{for } i<\min E_{2},\\
0 & \text{for } \min E_{2} \leq i \leq\max E_{3},\\
\mu_{i}+1 & \text{for } \max E_{3} \leq i \leq N+j.
\end{cases}
\end{gather*}
\end{Theorem}

Notice that the description is given in terms of $E_{2}$ and $E_{3}$ to avoid awkwardness with $E_{1}=\varnothing$ or $E_{4}=\varnothing$, when $k=\tau_{j+1}$ or $1$, respectively.

\begin{proof}
By Lemma~\ref{Betai} and Proposition~\ref{LenBeta1}, $\beta_{i}=1$ for $\nu_{1}+1\leq i\leq N+j$ and, by hypothesis, $\beta_{i}=0$ for $i\leq N+j$ if and only if $i\in E_{2}\cup E_{3}$. Thus, we consider the values of $\beta$ on $J= [1,\min E_{2}-1 ]\cup [\max E_{3}+1,\nu_{1} ]$.
First, we show that $d+1\leq\beta_{i}\leq d+K-1$, for $i\in J$. Suppose $i\in I_{u}$, with $I_{j}=E_{4}$ and $I_{j+1}=E_{1}$. Then, $R_{\mu}(i)=R_{\beta }(i)=i+n(d+u-2)$. Moreover, $i\in J$ implies that $r_{\beta}(i)\leq\nu_{1}-n$ and then, $i+n(d+u-2)\leq\nu_{1}-n+n\beta_{i}$. This last inequality translates into:
\begin{gather*}
n\beta_{i} \geq i +n(d+u-2)+n-\nu_{1} = (i-\nu_{u}) + \nu_{u}-\nu_{1} +n(d+u-1) \\
\hphantom{n\beta_{i}}{}=(i-\nu_{u}) + n(d+u-1) - (u-1)(n-1) = (i-\nu_{u})+nd+u-1.
\end{gather*}
Since $i-\nu_{u}\geq1$, we have that $\beta_{i}\geq d+1$.

Let $C_{i}= \{ s\colon \beta_{s}=d+i-2 \}$, for $i\leq K+1$, and $m_{i}=n-1-\#C_{i}$ for $3\leq i\leq K$ and $m_{K+1}=\nu_{K}-\#C_{K+1}$. By the maximum principle and the fact that $\mu\trianglerighteq\beta$, it follows that $m_{i}\geq0$, for all $i$. Note that the set $E_{2}\cup E_{3}$ is excluded here. There are two equations satisfied by the $m_{i}$'s:
\begin{gather*}
\sum_{i=3}^{K+1}\#C_{i} = \nu_{1}-n \qquad \text{and} \qquad
\sum_{i=3}^{K+1}\#C_{i}(d+i-2) = |\beta|-(j+nd-1).
\end{gather*}
Simplifying the first equation, we get that $\sum\limits_{i=3}^{K+1}m_{i}=1$, and simplifying the second equation, which requires more computation, we get that $\sum\limits_{i=3}^{k+1} m_{i}(d+i-2) =d+j-1$.

The unique solution is $m_{j+1}=1$ and $m_{i}=0$, for $i\neq j+1$, implying that $\#C_{i}=n-1$. Thus, $C_{i}=I_{g(i)}$, for some $g(i)\neq j,j+1$, and $C_{j+1}=E_{1}\cup E_{4}$, since $\#C_{j+1}=n-2$. The obvious modifications are made here if $\nu_{K}<n-1$ or $j=K$. If $i>j+1$, then $r_{\beta} (
\nu_{g(i)-1} )=\nu_{i-1}$ and $r_{\beta} (\nu_{g(i)-1} ) -r_{\mu} (\nu_{g(i)-1} )=n[(d+i-2)-d+g(i)-2)]$. Therefore, $n(i-g(i))=(\nu_{i-1}-\nu_{g(i)-1})=(n-1)(g(i)-i)$ and $i=g(i)$, thus $\beta_{u}=\mu_{u}$, for $u\in I_{g(i)}$. If $i<j$, then $r_{\beta}(\nu_{g(i)-1}) =\nu_{i-1}-n$ and $r_{\beta}(\nu_{g(i)-1})-r_{\mu}(\nu_{g(i)-1})=n[(d+i-2)-1-(d+g(i)-2)]$. Thus, $g(i)=i-1$ and $\beta_{u}=d+g(i)-1=\mu_{u}+1$, for $u\in I_{g(i)}$.
\end{proof}

It remains to show the other case described in Lemma~\ref{Betai}. The next results show that if $\beta_{\nu_{1}+1}=0$, then $\beta=\mu$. Let us start with a lemma.

\begin{Lemma}\label{meetE2E3} Let $c$ and $c^{\prime}$ be two different indexing parameters
such that their corresponding sets $B_{c}= \{ i\colon \beta_{i}=c \}$ and $B_{c^{\prime}}= \{ i\colon \beta_{i}=c^{\prime} \}$ satisfy that $B_{c}\cup B_{c^{\prime}}=I_{j}\cup I_{j+1}$, with $B_{c^{\prime}}\cap E_{2}\neq\varnothing$. Then, $B_{c}=E_{1}\cup E_{3}$ and $B_{c^{\prime}}=E_{2}\cup E_{4}$, or $B_{c}=E_{1}\cup E_{4}$ and $B_{c^{\prime}}=E_{2}\cup E_{3}$. Moreover, in the latter case, $\#B_{c^{\prime}}=n$.
\end{Lemma}

\begin{proof}If $\#B_{c}\leq n-2$, then $\#B_{c^{\prime}}\geq n$, which means that $B_{c^{\prime}}=E_{2}\cup E_{3}$.

The cases $\mathbf{b}=(1010)$ and $(1110)$ allow $\#B=n-1$ and imply that $E_{4}\subset B_{c^{\prime}}$ and that $\mathbf{b}^{\prime}=(0111)$ and $(0101)$, respectively.

The case $(1000)$ is excluded because $\#E_{1}=\tau_{j+1}-k\leq\tau_{j+1}-1$, as well as $\mathbf{b}^{\prime}=(0001)$ because $\#E_{4}=k-1\leq n-2$.
Finally, $\mathbf{b}=(1110)$ and $\mathbf{b}^{\prime}=(0111)$ can not occur because the state $(\ast11\ast)$ implies $t_{j+1}^{\prime}=s_{j}^{\prime}$ and $\beta_{\nu_{j-1}-2k+1}=c=\beta_{\nu_{j-1}-2k+2}=c^{\prime}$.
\end{proof}

We are ready to prove that $\beta=\mu$, under the conditions established for this case.

\begin{Theorem}\label{FinalThm2} If $R_{\beta}=R_{\mu}$, $\mu\trianglerighteq\beta$, $j>1$ and $\ell(\beta)\leq\nu_{1}$, then $\beta=\mu$.
\end{Theorem}

\begin{proof}Since $\mu\trianglerighteq\beta$, $d\leq\beta_{i}\leq d+K-1$ for $1\leq i\leq\nu_{1}$.

Let $C_{i}=\{ s\colon\beta_{s}=d+i-2\}$ and $m_{i}=\#C_{i}- (\nu_{i-1}-\nu_{i} )$, for $2\leq i\leq K+1$. Then $\sum\limits_{i=2}^{K+1}m_{i}=0$ and $\sum\limits_{i=2}^{K+1}m_{i}(d+i-2)=0$. We also have that $m_{K+1}\leq0$ because $\mu\trianglerighteq\beta$. For $2\leq i\leq K$, by the study about the size of $B$ presented in Section~\ref{SetB}, we know that $\nu_{i-1}-\nu_{i}=n-1$, and this implies that $m_{i}\leq1$ and that at most one value of $i$ allows $m_{i}=1$. Now, this is impossible because the sums would imply there exists $u$ such that $m_{u}=-1$ and $(d+i-2)-(d+u-2)=0$.

The previous argument shows that the level sets of $\beta$ are permutations of the level sets of $\mu$. That is, $C_{i}=I_{g(i)}$ for some $g(i)\neq j, j+1$. Since $\beta$ is a permutation of $\mu$, $r_{\beta}(\nu_{g(i)-1})=\nu_{i-1}$ and the rank equation gives $r_{\beta}(\nu_{g(i)-1})-r_{\mu}(\nu_{g(i)-1})=n[(d+i-2)-(d+g(i)-2)]$ and $n(i-g(i))=(\nu_{i-1}-\nu_{g(i)-1})=(n-1)(g(i)-i)$. We conclude then that $i=g(i)$.

As a consequence, $\#C_{j}=n-1=\#C_{j+1}$, or $\#C_{j+1}=\nu_{K}$ if $j=K$, and $C_{j}\cup C_{j+1}=I_{j}\cup I_{j+1}$. If $i\in C_{j}\cap(E_{2}\cup E_{4})$ or $i\in C_{j+1}\cap(E_{1}\cup E_{3})$, then $r_{\beta}(i)=r_{\mu}(i)$. Therefore, $i=\min B_{j}=\min E_{2}$ or else $r_{\beta}(i)-r_{\mu}(i)=\nu_{j}+1-i=n$ and $i=\nu_{j+1}\notin E_{1}$ when $i\in E_{1}$, or
$r_{\beta}(i)-r_{\mu}(i) =\nu_{j}+1-(i-n+k)=n$, $i\in E_{3}$ and $i=\nu_{j}+1-k\in E_{2}$. Thus $C_{j}\cap E_{2}\neq\varnothing$ and we apply Lemma~\ref{meetE2E3} taking $B_{c^{\prime}}=C_{j}$ and $B_{c}=C_{j+1}$ to conclude that $C_{j}=E_{2}\cup E_{4}$ and $C_{j+1}=E_{1}\cup E_{3}$. The case $C_{j}=E_{2}\cup E_{3}$ is impossible since $\#(E_{2}\cup E_{3})=n$. Thus $\beta=\mu$.
\end{proof}

We finish the case $j>1$ with the following theorem.

\begin{Theorem}For $j>1$, if $(\mu,\beta)$ is an $(1,n)$-critical pair, then $\ell (\beta)=N+j$.
\end{Theorem}

\subsection[The pairs $(\alpha(\Theta_{j,k}),\beta)$: Case $j=1$]{The pairs $\boldsymbol{(\alpha(\Theta_{j,k}),\beta)}$: Case $\boldsymbol{j=1}$}\label{Pairsj=1}

In this case, $\mu$ has length $\ell(\mu)=N-k+1$, with $k\leq n-1$. It turns out that $\beta$ in the critical pair $(\mu,\beta)$ is a permutation of $\lambda= \mu^{+}$, and differs from~$\mu$ only in the arrangement of the values~$d$ and~$0$. Moreover, $\ell(\beta)=N+1$.

The relevant subdivision of $I_{1}\cup I_{2}$ and its properties are:
\begin{itemize}\itemsep=0pt
\item $E_{1}= [\nu_{2}+1,\nu_{1}-k ]$ and $E_{1}=\varnothing$ if $k=\tau_{2}$. For $i\in E_{1}$, $\mu_{i}=d$ and $r_{\mu}(i)=i$. Moreover, $\# E_{1} = \tau_{2}-k$.

\item $E_{2}= [\nu_{1}-k+1,N-2k+1 ]$. For $i\in E_{2}$, $\mu_{i}=0$ and $r_{\mu}(i)=i+k$. Moreover, $\#E_{2}=nd-k$.

\item $E_{3}= [N-2k+2,N-k+1 ]$. For $i\in E_{3}$, $\mu_{i}=d$ and $r_{\mu}(i)=i-nd+k$. Moreover, $\#E_{3}=k$.

\item $E_{4}= [N-k+2,N ]$ and $E_{4}=\varnothing$ if $k=1$. For $i\in E_{4}$, $\mu_{i}=0$ and $r_{\mu}(i)=i$. Moreover, $\# E_{4}=k-1$.
\end{itemize}

Furthermore, the intervals $[a_{u},b_{u}]$ for $B\cap (I_{1}\cup I_{2})$ are of the form:
\begin{itemize}\itemsep=0pt
\item $B\cap E_{1}=[a_{2},b_{2}]$ and $[s_{2},t_{2}]=[\nu_{1}-b_{2},\nu_{1}-a_{2}]$, with $k\leq s_{2}\leq t_{2}\leq \tau_2-1$.
\item $B\cap E_{2}=[a_{1}^{\prime},b_{1}^{\prime}]$ and $[s_{1}^{\prime},t_{1}^{\prime}] =[N-k-b_{1}^{\prime},N-k-a_{1}^{\prime}]$, with $k-1\leq s_{1}^{\prime}\leq t_{1}^{\prime}\leq nd-2$.
\item $B\cap E_{3}=[a_{2}^{\prime},b_{2}^{\prime}]$ and $[s_{2}^{\prime}, t_{2}^{\prime}]= [N-k+1-b_{2}^{\prime}, N-k+1-a_{2}^{\prime}]$, with $0\leq s_{2}^{\prime}\leq t_{2}^{\prime}\leq k-1$.
\item $B\cap E_{4}=[a_{1},b_{1}]$ and $[s_{1},t_{1}]=[N-b_{1},N-a_{1}]$, with $0\leq s_{1}\leq t_{1}\leq k-2$.
\end{itemize}

Using the same \emph{shorthand notation}, the analysis of the sequence $\mathbf{b}$ depends on the intersection of $B$ with the $I_{u}$ intervals. Since the arguments for $j>1$ apply here when $E_{2}$ is not involved, we summarize the results that we can extend from the case $j>1$.

First, suppose $B\cap I_{u}\neq\varnothing$ and $B\cap I_{s}=\varnothing$ for $2<s<u$. This implies that $\tau_{2}=n-1$. The following result resumes part of the information we know about $\mathbf{b}$.

\begin{Proposition}Suppose $B\cap I_{u}\neq\varnothing$ and $B\cap I_{s}=\varnothing$ for $2<s<u$.
\begin{itemize}\itemsep=0pt
\item If $\mathbf{b}=(1\ast\ast\ast)$, then $s_{u}-t_{2}=u-1>1$ and $[s_{2},t_{2}]\ll[s_{u},t_{u}]$.
\item If $\mathbf{b}=(01\ast\ast)$, then $s_{u}-t_{1}^{\prime}=u>2$ and $[s_{1}^{\prime},t_{1}^{\prime}+1]\ll[s_{u},t_{u}]$.
\item If $\mathbf{b}=(001\ast)$, then $s_{u}-t_{2}^{\prime}=u-1>1$ and $[s_{2}^{\prime},t_{2}^{\prime}]\ll[s_{u},t_{u}]$.
\item If $\mathbf{b}=(0001)$, then $s_{u}-t_{1}=u>2$ and $[s_{1},t_{1}]\ll[s_{u},t_{u}]$.
\end{itemize}
\end{Proposition}

There a few more configurations for which we know more details.

\begin{Proposition}\label{Prop:Casesb18}\quad
\begin{itemize}\itemsep=0pt
\item For $\mathbf{b}=(11\ast\ast)$, $s_{2}=t_{1}^{\prime}+2$ and $[s_{1}^{\prime},t_{1}^{\prime}]\ll[s_{2},t_{2}]$.
\item For $\mathbf{b}=(\ast11\ast)$, $t_{2}^{\prime}
=s_{1}^{\prime}$ and $\#B\cap(E_{2}\cup E_{3})=\#[s_{2}^{\prime},t_{1}^{\prime}]+1$. Notice that this implies also that $a_{2}^{\prime}=b_{1}^{\prime}+1=N-2k+2$, so this configuration is possible for only one value of $c$.
\item For $\mathbf{b}=(\ast\ast11)$, $s_{2}^{\prime}=t_{1}+2$ and $[s_{1},t_{1}] \ll[s_{2}^{\prime},t_{2}^{\prime}]$.
\item For $\mathbf{b}=(101\ast)$, $a_{2}^{\prime}-b_{2}=nd-k+1=\#E_{2}+1$. Then, $b_{2}=\nu_{1}-k$, $a_{2}^{\prime}=N-2k+2$ and $s_{2}=t_{2}^{\prime}+1$. Therefore, $\#(B\cap(E_{1}\cup E_{3}))=\#[s_{2}^{\prime},t_{2}]$.
\item For $\mathbf{b}=(\ast101)$, $a_{1}-b_{1}^{\prime}=k+1=\#E_{3}+1$, thus $b_{1}^{\prime}=N-2k+1$, $a_{1}=N-2k+2$ and $s_{1}^{\prime}=t_{1}+1$. Therefore, $\#(B\cap(E_{2}\cup E_{4}))=\#[s_{1},t_{1}^{\prime}]$.
\item For $\mathbf{b}=(1001)$, $s_{2}=t_{1}+2$, $[s_{1},t_{1}]\ll[s_{2},t_{2}]$.
\end{itemize}
\end{Proposition}

From these relations it follows that if $B\cap I_{u}\neq\varnothing$ and that $B\cap(I_{1}\cup I_{2})=[\tilde{s},\tilde{t}]\neq\varnothing$, with $[\tilde{s},\tilde{t}]\ll[s_{u},t_{u}]$. Then, we have the following result.

\begin{Corollary}Suppose $B\cap I_{u}\neq\varnothing$ and $B\cap I_{s}=\varnothing$ for $2<s<u$. Then, for $j=1$, $\#B\leq n-2$ and $\#B\leq\nu_{K}-1$, if $B\cap I_{K+1}\neq\varnothing$.
\end{Corollary}

Now, suppose $B\cap I_{u}=\varnothing$, for all $u>2$. The following bounds are combinations of the relations among the intervals stated in Proposition~\ref{Prop:Casesb18}.

\begin{Proposition}\label{Prop:CasesSizeB}\quad
\begin{itemize}\itemsep=0pt
\item For $\mathbf{b}=(\ast0\ast0)$, $B\subset E_{1}\cup E_{3}$ and $\#B\leq n-1$.
\item For $\mathbf{b}=(1110)$, $\#B\leq n-1$.
\item For $\mathbf{b}=(1111)$, $\#B\leq n-2$.
\item For $\mathbf{b}=(100\ast)$ and $(\ast001)$, $\#B\leq n-2$ since $\#B\leq\#[s_{1},t_{1}] +\#[s_{2},t_{2}]$ with $0\leq s_{1}\leq t_{1}\leq k-2$ and $k\leq s_{2}\leq t_{2}\leq n-2$.
\item For $\mathbf{b}=(0\ast0\ast)$, $B\subset E_{2}\cup E_{4}$ and $\#B\leq dn-1$.
\item For $\mathbf{b}=(0111)$, $\#B\leq dn-1$.
\item For $\mathbf{b}=(0110)$, $\#B\leq dn$.
\end{itemize}
\end{Proposition}

\begin{Lemma}\label{noteq0}Set $i_{0}=\min E_{2}=\nu_{1}-k+1$. If $i<i_{0}$, then $\beta_{i}\neq\beta_{i_{0}}$.
\end{Lemma}

\begin{proof}Consider $B=\{i\colon 1\leq i\leq N,\beta_{i}=\beta_{i_{0}}\}$. By
definition and the fact that $a_{1}^{\prime}=i_{0}$, it follows that $t_{1}^{\prime}=nd-2$. If $B\cap E_{1}\neq\varnothing$, then $\mathbf{b}= (11\ast\ast )$ and the inequality $ [s_{2}^{\prime},t_{2}^{\prime} ]\ll[s_{1},t_{1}]$ holds. However, this implies that $s_{1}\leq n-2$, which is contrary to $s_{1}\geq t_{2}^{\prime}+2$. If $B\cap I_{u}\neq\varnothing$ for some $u>2$, then $[s_{1}^{\prime},t_{1}^{\prime}+1]\ll[s_{u},t_{u}]$ because $s_{u}\leq n-2$.
\end{proof}

The goal of the remaining discussion is to show that either $\beta=\mu$ or $\ell(\beta)=N+1$ and $\beta$ is the unique solution of $R_\beta = R_\mu$ and $\mu\triangleright \beta$.
\begin{Proposition}\label{LenBeta} Consider $i_{0}=\min E_{2}=\nu_{1}-k+1$. Then, $\beta_{i_{0}}=0$ and $\ell(\beta)\leq N+1$.
\end{Proposition}

\begin{proof}Recalling that $R_{\beta}=R_{\mu}$ and noticing that $\mu_{i_{0}}=0$, we have that $R_{\beta}(i_{0})=\nu_{1}+1$. Let $b=\beta_{i_{0}}$. Thus $r_{\beta}(i_{0})=\nu_{1}+1-nb$. First we show $b\leq1$. Consider $B_{s}=\{i\colon 1\leq i\leq N,\beta_{i}=s\}$, for $s\geq0$. By Lemma~\ref{noteq0},
\begin{align*}
\nu_{1}+1-nb & =r_{\beta}(i_{0})=1+\#\{i\colon 1\leq i\leq N,\, \beta_{i}>b\}+\#\{i\colon i>N,\, \beta_{i}>b\} \\
& \geq N+1-\sum_{s=0}^{b}\#B_{s}\geq N+1-nd-b(n-1)=\nu_{1}-b(n-1).
\end{align*}
The cardinalities in the formula for $r_{\beta} (i_{0})$ are computed by changing to the set-theoretic complement; then $b$ of the numbers $\#B_{s}$ satisfy $\#B_{s}\leq n-1$ and by Proposition~\ref{Prop:CasesSizeB}, at most one satisfies $\#B_{s}\leq dn$. Thus $0\leq b\leq1$. The only possibility for $\#B_{s}=nd$ is $s=b$. If this bound is not achieved then $\#B_{s}=nd-1$ or $n-1$ where the bound $nd-1$ is possible only once, and the inequality becomes $r_{\beta}(i_{0})\geq\nu_{1}+1-b(n-1)$, implying that $b=0$.

Suppose $b=1$, then $B_{1}=E_{2}\cup E_{3}$ and $\beta_{i}\neq1$ for $i\in\tilde{E}=[1,\min E_{2}-1]\cup[\max E_{3}+1,N]$. Thus $B_{0}\subset\tilde{E}$. Set $m_{>1}=\# \{ i>N\colon \beta_{i}>1 \}$. Then,
\begin{gather}\label{Eq1}
\nu_{1}+1-n=r_{\beta}(i_{0})=1+\#\big\{i\in\tilde{E}\colon \beta_{i}>1\big\}+m_{>1}=1+\tilde{E}-\#B_{0}+m_{>1},
\end{gather}
and $\#B_{0}=N-nd+n-\nu_{1}+m_{>1}=n-1+m_{>1}$.

The equation~\eqref{Eq1} is only possible if $m_{0}=n-1$ and $m_{>1}=0$. Therefore, $\beta_{N+1}\leq1$ and $B_{0}=I_{u}$, for some $u>2$. This implies that $\beta_{i}>0$, for $\nu_{u-1}<i\leq N$, and $\beta_{i}>0$ also for $N<i\leq\ell(\beta)$. Thus, $r_{\beta}(\nu_{u-1})=\ell(\beta)$, $R_{\beta}(\nu_{u-1})=\ell(\beta)=R_{\mu}(\nu_{u-1})=\nu_{u-1}+n(d+u-2)=N+u-1$, and $\ell(\beta)=N+u-1>N$. However, but $\beta_{N+1}=0$ else $R_{\mu}(N+1)=N+1$ and, since $\# \{ i\leq N+1\colon \beta_{i}\geq1 \}=(N+1)-(n-1)$, $R_{\beta}(N+1)=N+2-n+n\beta_{N+1}=N+2$. Then, $\ell(\beta)\leq N$ and we get
to a contradiction.

This proves that $b=0$. Thus $r_{\beta}(i_{0})=\nu_{1}+1$ and $\# \{i\colon \beta_{i}>0 \}=\nu_{1}$. Consider $m_{>N}=\# \{i\colon i>N,\beta_{i}>0 \}$. Then, $\nu_{1}+1=N+1-\#B_{0}+m_{>N}\geq N+1-nd+m_{>N}=\nu_{1}+m_{>N}$, which means $m_{>N}\leq1$. If $m_{>N}=0$, then $\ell(\beta)\leq N$ and $\#B_{0}=nd-1$, otherwise $\ell(\beta)=N+1$ and $\#B_{0}=nd$.
\end{proof}

This is rather a complicated argument but it is a key step in the development of our study.

\begin{Corollary}\label{Cor52} For $i\geq\min E_{2}$, $\beta_{i}=0$. If $i<\min E_2$, the $\#\{i\colon \beta_{i}>0\} =\nu_{1}$ and $d\leq\beta_{i}\leq d+K-1$.
\end{Corollary}

\begin{proof} The bounds $d\leq\beta_{i}\leq d+K-1$ for the nonzero values follow from $\#\{i\colon \beta_{i}>0\}=\nu_{1}=\#\{i\colon \mu_{i}>0\}$ and $\mu\trianglerighteq\beta$. Finally, by Lemma~\ref{noteq0}, $\beta_{i}>0$, for $i<\min E_{2}$.
\end{proof}

Our next result is a first step in the direction of Theorems~\ref{FinalThm1} and~\ref{FinalThm2}.

\begin{Proposition}If $j=1$, $\mu\trianglerighteq\beta$ and $R_{\mu}=R_{\beta}$, then $\beta^{+}=\lambda=\mu^{+}$.
\end{Proposition}

\begin{proof} From Corollary~\ref{Cor52}, $d\leq\beta_{i}\leq d+K-1$ or $\beta_{i}=0$. Let $C_{i}= \{ s\colon \beta_{s}=d+i-2 \}$ and $m_{i}=\#C_{i}- (\nu_{i-1}-\nu_{i} )$, for $2\leq i\leq K+1$. It is possible that $N+1\in C_{i}$ for some $i$ only if $ \{ i\colon 1\leq i\leq N,\beta_{i}=0 \}=E_{2}\cup E_{3}$ and $\#C_{u}\leq n-1$ for $u\ne q i$. The equations $\sum\limits_{i=2}^{K+1}m_{i}=0$ and $\sum\limits_{i=2}^{K+1}m_{i}(d+i-2)=0$, together with the bound $m_{K+1}\leq0$ (since $\mu\trianglerighteq\beta$), implies that $m_{i}\leq 0$. The value $m_{i}=1$ is impossible because the sums would imply there exists $u$ such that $m_{u}=-1$ and $(d+i-2)-(d+u-2)=0$. Therefore, the bound $n-1$ applies to all the sets $C_{u}$. From Proposition~\ref{LenBeta}, $\{ i\colon i\leq N,\, \beta_{i}=0\}=E_{2}\cup E_{3}$ or $E_{2}\cup E_{4}$. Thus $\#C_{i}=n-1$, for $2\leq i\leq K$, and $\#C_{K+1}=\nu_{K}$. Equivalently $\beta^{+}=\lambda$.
\end{proof}

\begin{Proposition}If $u>2$, then $\beta_{i}=\mu_{i}$ for $i\in I_{u}$.
\end{Proposition}

\begin{proof}We proceed as in the proof of Proposition~\ref{beta1}.

Consider $R_{\mu}(\nu_{K})=\nu_{K}+n(d+K-1)=r_{\beta}(\nu_{K})+n\beta_{\nu_{K}}$. Suppose $\beta_{\nu_{K}}=d+u-2$ and $u<K+1$. Notice that the value $\beta_{\nu_{K}}=0$ can not occur by Corollary~\ref{Cor52}. Since $\beta^{+}=\lambda$, $\# \{ s\colon \beta_{s}>d+u-2\}=\nu_{u}$. Then, by definition,
\begin{align*}
r_{\beta}(\nu_{K})& =\# \{ s\colon s\leq\nu_{K},\,\beta_{s}=d+u-2 \}+\# \{ s\colon \beta_{s}>d+u-2 \}\\
& =\# \{ s\colon s\leq\nu_{K},\,\beta_{s}=d+u-2 \}+\nu_{u}\leq\nu_{K}+\nu_{u},
\end{align*}
and thus $\nu_{K}+n(d+K-1)-n(d+u-2)\leq\nu_{K}+\nu_{u}$, which simplifies to $n+K-u\leq\nu_{K}$ and implies $K-u\leq\nu_{K}-n\leq-1$.

The bound $u\geq K+1$ contradicts that $\beta_{\nu_{K}}<\mu_{\nu_{K}}$.
Therefore, $\beta_{\nu_{K}}=d+K-1$ and $r_{\beta}(\nu_{K})=\nu_{K}$. The hypothesis $\mu\trianglerighteq\beta$ implies $\beta_{i}\leq d+K-1$ and so, $\beta_{i}=d+K-1$ for $1\leq i\leq\nu_{K}$.

Arguing inductively, suppose $\beta_{i}=\mu_{i}$ for $i\in\bigcup\limits_{s=u+1}^{K+1}I_{s}$. The possible nonzero values of $\beta$ on $I_{u}$ are $d+s-2$ for $2\leq s\leq u$. Consider $R_{\beta}(\nu_{u-1})=R_{\mu}(\nu_{u-1})=\nu_{u-1}+n(d+u-2)$. A~similar argument shows $\beta_{i}=d+u-2$, for $i\in I_{u}$.
\end{proof}

It remains to consider the set $\{ i\colon \beta_{i}=d\}$.

\begin{Lemma}\label{LastLemma} Consider $B_{d}= \{ i\colon i\leq N,\, \beta_{i}=d \}$ and $B_{0}= \{ i\colon i\leq N,\, \beta_{i}=0 \}$ such that $B_{d}\cup B_{0}=I_1\cup I_2$. Then, $B_{d}=E_{1}\cup E_{3}$ and $B_{0}=E_{2}\cup E_{4}$, or $B_{d}=E_{1}\cup E_{4}$ and $B_{0}=E_{2}\cup E_{3}$.
\end{Lemma}

\begin{proof}If $\#B_{d}\leq n-2$, then $\#B_{0}\geq dn$ implying $B_{0}=E_{2}\cup E_{3}$. Let $\mathbf{b}$ be the configuration for~$B_{d}$ and $\mathbf{b}^{\prime}$ be the configuration for~$B_{0}$.

The cases $\mathbf{b}=(1010)$ and $(1110)$ allow $\#B_{d}=n-1$ and imply $E_{4}\subset B_{0}$ and $\mathbf{b}^{\prime}=(0111)$ or $(0101)$.

The case $(1000)$ is excluded because $\#E_{1}=\tau_{2}-k\leq\tau_{2}-1$, as well as the case $(0001)$, because $\#E_{4}=k-1\leq n-2$. Finally, $\mathbf{b}=(1110)$ and $\mathbf{b}^{\prime}=(0111)$ can not occur because the state $(\ast11\ast)$ implies $t_{2}^{\prime}=s_{1}^{\prime}$ and $\beta_{N-2k+1}=d=\beta_{N-2k+2}=0$.
\end{proof}

We are ready to prove the analogous result to Theorems~\ref{FinalThm1} and~\ref{FinalThm2} for $j=1$.

\begin{Theorem}If $j=1$, $\mu\trianglerighteq\beta$ and $R_{\mu}=R_{\beta}$ then either $\beta=\mu$ or $\ell(\beta)=N+1$ and $\beta$ is unique.
\end{Theorem}

\begin{proof}By Lemma~\ref{LastLemma}, if $\{i\colon i\leq N,\,\beta_{i}=d\}=E_{1}\cup E_{3}$ and $\{i\colon i\leq N,\,\beta_{i}=0\}=E_{2}\cup E_{4}$, then $\beta=\mu$. Otherwise $\{ i\colon \beta_{i}=d\}=E_{1}\cup E_{4}\cup\{ N+1\}$, which has cardinality $n-1$, and $\{i\colon i\leq N,\beta_{i}=0\}=E_{2}\cup E_{3}$, with cardinality~$nd$.
\end{proof}

We finish this section illustrating our results with an example.
\begin{example*}
Consider the parameters $j=1$, $K=2$, $n=4$, $d=3$, and $N=17$, for which $\mu=\big(4,4,4,3,0^{10},3,3,0\big)$ and $\beta=\big(4,4,4,3,0^{12},3,3\big)$, with $\ell(\beta)=18$.
\end{example*}

\section{Concluding remarks}\label{Sec7}

We have shown that if $\S \in\mathsf{Tab}_{\tau}$ with $\mathsf{col}_{\mathbb{S}}[i] = \mathsf{col}_{\mathbb{S}}[i+1]=k$ and $\mathsf{row}_{\mathbb{S}}[i+1]=j$, then the polynomials $M_{\alpha (\S s_{i} )}$ and $M_{\alpha (\Theta_{j,k} )}$ are $\varpi$-equipolar for
$S\in\mathsf{Tab}_{\tau}$ and $M_{\alpha (\S s_{i} )}$ has no pole at~$\varpi$ in~$N$ variables. Hence, the polynomials $M_{\alpha(\S)}$, for $S\in\mathsf{Tab}_{\tau}$ specialized to~$\varpi$ satisfy the equations $M_{\alpha(\S)}\xi_{i}=M_{\alpha (\S)}\phi_{i}$ for all~$i$, and are singular.

The result on critical pairs provides a new proof for singular nonsymmetric Jack polynomials with the restriction $\gcd (m,n)=1$; then the quasistaircase polynomials are singular for $\kappa=-\frac{m}{n}$ (see~\cite{D2005}). Considering the known singular nonsymmetric Jack polynomials theory we suspect that there are no singular Macdonald polynomials other than the quasistaircase types constructed in this paper. This may be quite harder to prove, if true.

We also want to point out that there is a different behavior for partitions with only two parts, in the sense that there may be more than one quasistaircase for a given $\tau$ and $(m,n)$.
That is, consider the case $K=1$, for which necessarily $j=1$. Then, $\tau= (N-\tau_{2},\tau_{2})$ with $\tau_{2}\leq N/2$, and $\lambda=\big(m^{\tau_{2}},0^{N-\tau_{2}}\big)$, with $n=N-\tau_{2}+1$. We want to figure out the values of $\omega$ for which $\varpi=\big(\omega u^{-n},u^{m}\big)$ provides singular polynomials. Let $g=\gcd (m,n )$ and $d$ be a factor of~$g$.

To produce a quasistaircase, set $n=dn_{1},m=dm_{1}$ subject to $\tau_{2}\leq n_{1}-1$. That is, $\frac{n}{d}\geq\tau_{2}+1$, or $d\leq \frac{n}{\tau_{2}+1}$. Then, let $\omega=\exp\big(\frac{2\pi k\mathrm{i}}{m}\big)$ with $\gcd(g,k)=d$. As a result $(q,t)=\big(\omega u^{-n},u^{m}\big)$ satisfies
$q^{m/d}t^{n/d}=1$. This formula is based on replacing $m$, $n$, and $g$ by $\frac{m}{d}$, $\frac{n}{d}$, and $\frac{g}{d}$, respectively, and setting $k=k^{\prime}d$, with $\gcd\big(k^{\prime},\frac{g}{d}\big)=1$.

We wrap up the paper with a last example illustrating all the study done here.
\begin{example*}
Let $\lambda=\big(30^{3},0^{11}\big)$, for which $N=14$, $n=12$, and $\tau_{2}=3$. Then, $\gcd(30,12)=6$ and $d$ is a factor of $6$ such that $d\leq\dfrac{12}{4}=3$. Thus, $\omega=\exp\big(\frac{2\pi k\mathrm{i}}{30}\big)$, with $\gcd (k,6)=1$, 2, or 3, resulting in the singular values $q^{30}t^{12}=1$, $q^{15}t^{6}=1$, and $q^{10}t^{4}=1$. In terms of $\varpi$ the implication is that $\varpi=\big(\omega u^{-2},u^{5}\big)$ where
\begin{enumerate}\itemsep=0pt
\item[1)] $\omega^{2}-\omega+1=0$ (primitive $6^{\rm th}$ root of unity) and $q^{30}t^{12}=1$;
\item[2)] $\omega^{2}+\omega+1=0$ (primitive $3^{\rm rd}$ root of unity) and $q^{15}t^{6}=1$;
\item[3)] $\omega+1=0$ (primitive square root of unity) and $q^{10}t^{4}=1$.
\end{enumerate}

Note that the fact that $\omega=1$ is specifically excluded is a manifestation of the result that the nonsymmetric Jack polynomial with label $\big(30^{3},0^{11}\big)$ is not singular for $\kappa=-30/12=-5/2$, and so it may have poles. The known results in~\cite{D2005} assert that for every pair
$(m,n)$ with $2\leq n\leq14$ and $m=1, 2, 3,\ldots$ such that $\dfrac{m}{n}\notin\mathbb{Z}$ there is a nonsymmetric Jack polynomial singular for $\kappa=-\frac{m}{n}$. In our case, for the pair $(30,12)$ the corresponding label is $\big(40,35,30,0^{11}\big)$.
That is, $\big(30^{3},0^{11}\big)$ is not a valid label for singular Jack polynomials.
\end{example*}

\subsection*{Acknowledgements}

The authors would like to thank Jean-Gabriel Luque for his fruitful discussions and his colla\-bo\-ration during the previous years. They also thank the referees for their careful reading and suggestions on improving the presentation.

\pdfbookmark[1]{References}{ref}
\LastPageEnding

\end{document}